\documentclass[11pt,usenames,dvipsnames]{extarticle}

\usepackage{amsfonts}
\usepackage{amsmath}
\usepackage{amssymb} 
\usepackage{amscd}
\usepackage{amsthm}
\usepackage{mathrsfs}
\usepackage{graphicx}
\usepackage{wasysym}
\usepackage{enumerate}
\usepackage{xcolor,colortbl}
\usepackage{tikz,tikz-cd}
\usepackage{geometry}
\usepackage{physics}
\usepackage[12hr]{datetime}
\usepackage{textcomp}
\usepackage[backref=page]{hyperref}
\usepackage{inputenc}
\usepackage{appendix}
\usepackage{mathtools}
\usepackage{xfrac}
\usepackage{nicefrac}
\usepackage{url}
\usepackage[linesnumbered,ruled,vlined]{algorithm2e}

\hypersetup{
colorlinks=true, urlcolor=NavyBlue, linkcolor=Mahogany, citecolor=ForestGreen, pdfborder={0,0,0},
}

\mathchardef\mhyphen="2D

\newcommand{\mb}{\mathbb}
\newcommand{\ms}{\mathscr}
\newcommand{\mc}{\mathcal}

\newcommand{\tbf}{\textbf}
\newcommand{\tsf}{\textsf}

\newcommand{\msf}{\mathsf}
\newcommand{\n}{\enspace}
\newcommand{\tx}{\text}
\newcommand{\ol}{\overline}
\newcommand{\ul}{\underline}

\newcommand{\spn}{\tx{\normalfont span}}

\newcommand{\wgt}{\mathrm{wt}}
\newcommand{\poly}{\tx{poly}}

\newcommand{\cl}{\mathrm{cl}}
\newcommand{\zcl}{\tx{\normalfont Z-cl}}
\newcommand{\zscl}{\tx{\normalfont Z*-cl}}
\newcommand{\hcl}{\tx{\normalfont h-cl}}
\newcommand{\hscl}{\tx{\normalfont h*-cl}}

\newcommand{\iref}[2]{(\hyperref[#2]{\ref*{#1}.\ref*{#2}})}

\newcommand{\modulo}[1]{\,(\tx{\normalfont mod }#1)}
\newcommand{\Hilbert}{\mathrm{H}}
\newcommand{\certdeg}{\mathrm{cert\mhyphen deg}}

\newcommand{\EHC}{\msf{EHC}}
\newcommand{\EPC}{\msf{EPC}}
\newcommand{\HC}{\msf{HC}}
\newcommand{\PHC}{\msf{PHC}}
\newcommand{\PC}{\msf{PC}}
\newcommand{\PPC}{\msf{PPC}}

\newcommand{\email}[1]{\href{mailto:#1}{\textcolor{NavyBlue}{\texttt{#1}}}}

\newcommand{\SUt}{{\normalfont SU}\(^2\)}

\newcommand{\sqbinom}{\genfrac{[}{]}{0pt}{}}

\theoremstyle{theorem}
\newtheorem{theorem}{Theorem}[section]
\newtheorem{fact}[theorem]{Fact}

\newtheorem{proposition}[theorem]{Proposition}
\newtheorem{observation}[theorem]{Observation}
\newtheorem{lemma}[theorem]{Lemma}
\newtheorem{corollary}[theorem]{Corollary}
\newtheorem{conjecture}[theorem]{Conjecture}

\newtheorem{problem}[theorem]{Problem}
\newtheorem{openproblem}[theorem]{Open Problem}

\newtheoremstyle{TheoremNum}
{\topsep}{\topsep}              
{\itshape}                      
{}                              
{\bfseries}                     
{.}                             
{ }                             
{\thmname{#1}\thmnote{ \bfseries #3}}
\theoremstyle{TheoremNum}
\newtheorem{reptheorem}{Theorem}

\SetKwInput{KwInput}{Input}                
\SetKwInput{KwOutput}{Output}              

\title{Covering Symmetric Sets of the Boolean Cube\\by Affine Hyperplanes}
\author{S. Venkitesh\footnote{Department of Mathematics, IIT Bombay, Powai, Mumbai, India.  Email: \email{venkitesh.mail@gmail.com}.  Website: \url{https://sites.google.com/view/venkitesh}.  Supported by the Senior Research Fellowship of CSIR, HRDG, Government of India.}}
\date{}

\begin{document}
	
	\maketitle
	
	\begin{abstract}
		Alon and F\"uredi~(European J. Combin., 1993) proved that any family of hyperplanes that covers every point of the Boolean cube \(\{0,1\}^n\) except one must contain at least \(n\) hyperplanes.  We obtain two extensions of this result, in characteristic zero, for hyperplane covers of symmetric sets of the Boolean cube (subsets that are closed under permutations of coordinates), as well as for \emph{polynomial covers} of \emph{weight-determined} sets of \emph{strictly unimodal uniform} (\SUt) grids.
		
		As a central tool for solving our problems, we give a combinatorial characterization of (finite-degree) Zariski (Z-) closures of symmetric sets of the Boolean cube.  In fact, we obtain a characterization that concerns, more generally, weight-determined sets of \SUt\,grids.  However, in this generality, our characterization is not of the Z-closures but of a new variant of Z-closures defined exclusively for weight-determined sets, which coincides with the Z-closures in the Boolean cube setting, for symmetric sets.  This characterization admits a linear time algorithm, and may also be of independent interest.  Indeed, as further applications, we\n(i) give an alternate proof of a lemma by Alon et al.~(IEEE Trans. Inform. Theory, 1988), and\n(ii) characterize the \emph{certifying degrees} of weight-determined sets.
		
		In the Boolean cube setting, our above characterization can also be derived using a result of Bernasconi and Egidi (Inf. Comput., 1999) that determines the affine Hilbert functions of symmetric sets.  However, our proof is independent of this result, works for all \SUt\,grids, and could be regarded as being more combinatorial.
		
		We also introduce another new variant of Z-closures to understand better the difference between the hyperplane and polynomial covering problems over uniform grids.  We conclude by introducing a third variant of our covering problems and solving it\footnote{The solution to the third variant in the Boolean cube setting was conjectured in the published version of this work.  Later, the conjecture could be disproved and the problem could be solved using a construction in~\cite{ghosh-kayal-nandi-2022-alon-furedi}.  The solution is included in Section~\ref{sec:updates}, containing all the updates to the published version.} in the Boolean cube setting.
	\end{abstract}
	
	\paragraph{Notations.}  \(\mb{R}\) denotes the set of all real numbers, \(\mb{Z}\) denotes the set of all integers, \(\mb{N}\) denotes the set of all nonnegative integers, and \(\mb{Z}^+\) denotes the set of all positive integers.
	
	\section{Introduction and overview}\label{sec:intro}
	
	We will work over the field \(\mb{R}\).  For any two integers \(a,b\in\mb{Z},\,a\le b\), by abuse of notation, we will denote the \emph{interval} of all integers between \(a\) and \(b\) by \([a,b]\).  Further, the interval of integers \([1,n]\) will also be denoted by \([n]\).  By a \tsf{uniform grid}, we mean a finite grid of the form \([0,k_1-1]\times\cdots\times[0,k_n-1]\), for some \(k_1,\ldots,k_n\in\mb{Z}^+\).  (This means for each \(i\in[n]\), the values taken by the points of the grid in the \(i\)-th coordinate are equispaced.)  Consider a uniform grid \(G=[0,k_1-1]\times\cdots\times[0,k_n-1]\).  For any \(x=(x_1,\ldots,x_n)\in G\), define the \tsf{weight} of \(x\) as \(\wgt(x)=\sum_{i\in[0,n]}x_i\).  We define a subset \(S\subseteq G\) to be \tsf{weight-determined} if
	\[
	x\in S,\,y\in G,\,\wgt(y)=\wgt(x)\quad\implies\quad y\in S.
	\]
	Let \(N=\sum_{i\in[0,n]}(k_i-1)\).  It follows that there is a one-to-one correspondence between weight-determined sets of \(G\) and subsets of \([0,N]\) -- a subset \(E\subseteq[0,N]\) corresponds to a weight-determined set \(\ul{E}\subseteq G\), defined as the set of all elements \(x\in G\) satisfying \(\wgt(x)\in E\).  When \(E=\{j\}\), a singleton set, we will denote \(\ul{E}\) by \(\ul{j}\).  Further, we will freely use this one-to-one correspondence and identify the weight-determined set \(\ul{E}\) with \(E\) without mention, whenever convenient.  This identification will be clear from the context.  In addition, for \(E\subseteq[0,N]\), we will denote \(|\ul{E}|\) by \(\sqbinom{G}{E}\).	It is immediate that \(\sqbinom{G}{j}=\sqbinom{G}{N-j}\), for all \(j\in[0,N]\).
	
	We say the uniform grid \(G\) is \tsf{strictly unimodal} if
	\[
	\sqbinom{G}{0}<\cdots<\sqbinom{G}{\lfloor N/2\rfloor}=\sqbinom{G}{\lceil N/2\rceil}>\cdots>\sqbinom{G}{N}.
	\]
	We will abbreviate the term `strictly unimodal uniform grid' by `\SUt\,grid'.  The \SUt\,condition on uniform grids is not very restrictive; there are enough interesting uniform grids which are \SUt.  For instance, the uniform grid \([0,k-1]^n\) is \SUt.  There is a simple characterization of the \SUt\,grids in terms of their dimensions, given by Dhand~\cite{DHAND201420} (stated in our work as Theorem~\ref{thm:strict-unimodal}).
	
	We will stick with the above notations whenever we consider uniform grids.  Further, we will assume throughout that \(k_i\ge2\), for all \(i\in[n]\).
	
	\paragraph{The Boolean cube setting.}  Consider the Boolean cube \(\{0,1\}^n\).  In this case, for any \(x\in\{0,1\}^n\), the weight \(\wgt(x)\) is equal to \(|x|\), the Hamming weight of \(x\), and \(\sqbinom{\{0,1\}^n}{j}=\binom{n}{j},\,j\in[0,n]\).  It is easy to check that \(\{0,1\}^n\) is strictly unimodal.  We define \(S\subseteq\{0,1\}^n\) to be \tsf{symmetric} if \(S\) is closed under permutations of coordinates.	It follows quite trivially that a subset is weight-determined if and only if it is symmetric.  This is not true though for any other uniform grid.  The Boolean cube \(\{0,1\}^n\subseteq[0,k_1-1]\times\cdots\times[0,k_n-1]\) is clearly symmetric but not weight-determined, if any of the \(k_i\)-s is at least 3.  Also without loss of generality, if \(k_i<k_j\) for some \(i,j\in[n],\,i\ne j\), then trivially the grid \([0,k_1-1]\times\cdots\times[0,k_n-1]\) is weight-determined but not symmetric.  In the case all the \(k_i\)-s are equal, every weight-determined set is indeed symmetric.

	\subsection{Our hyperplane and polynomial covering problems}\label{subsec:covering-problems}
	
	The term \emph{hyperplane covering problem} is commonly used in the literature (see Subsection~\ref{subsec:related}) to refer to any problem of finding the minimum number of hyperplanes covering a finite set in a finite-dimensional vector space (over a field) while satisfying some conditions.  Borrowing this terminology, we use the term \emph{polynomial covering problem} to refer to any problem of finding the minimum degree of a polynomial covering\footnote{We say a polynomial \(P(X_1,\ldots,X_n)\) covers a point \((a_1,\ldots,a_n)\) if \(P(a_1,\ldots,a_n)=0\).} a finite set in a finite-dimensional vector space (over a field) while satisfying some conditions.
	
	We will denote the polynomial ring \(\mb{R}[X_1,\ldots,X_n]\) by \(\mb{R}[\mb{X}]\).  Further, for any \(\alpha\in\mb{N}^n\), we denote the monomial \(X_1^{\alpha_1}\cdots X_n^{\alpha_n}\) by \(\mb{X}^\alpha\).
	
	Alon and F\"uredi~\cite{alon-furedi} considered the following hyperplane covering problem: what is the minimum number of hyperplanes required to cover every point of the Boolean cube \(\{0,1\}^n\), except the origin \(0^n\coloneqq(0,\ldots,0)\)?  They proved a lower bound of \(n\), which was clearly tight -- the family of \(n\) hyperplanes defined by the polynomials \(\sum_{j=1}^nX_j-i\), for \(i\in[n]\), satisfies the required conditions.  In fact, they proved the following stronger result.
	\begin{theorem}[{\cite[Theorem 1]{alon-furedi}}]\!\!\footnote{In~\cite{alon-furedi}, this result was proven true over any field.}\label{thm:alon-furedi-basic}
		If \(P(\mb{X})\in\mb{R}[\mb{X}]\) is a polynomial and \(a\in[0,k_1-1]\times\cdots\times[0,k_n-1]\) such that \(P(x)=0\), for all \(x
		\in[0,k_1-1]\times\cdots\times[0,k_n-1],\,x\ne a\) and \(P(a)\ne0\), then \(\deg P\ge\sum_{i\in[n]}(k_i-1)\).
	\end{theorem}
	
	In this work, we will consider extensions of this result to weight-determined sets of a uniform grid.  Let \(G\) be a uniform grid.  For a weight-determined set \(\ul{E}\), where \(E\subsetneq[0,N]\), we say a family of hyperplanes \(\mc{H}\) in \(\mb{R}^n\) is
	\begin{itemize}
		\item  a \tsf{nontrivial hyperplane cover} of \(\ul{E}\) if
		\[
		\ul{E}\subseteq G\cap\bigg(\bigcup_{H\in\mc{H}}H\bigg)\ne G.
		\]
		\item  a \tsf{proper hyperplane cover} of \(\ul{E}\) if
		\[
		\ul{E}\subseteq\bigcup_{H\in\mc{H}}H\quad\tx{and}\quad\ul{j}\not\subseteq\bigcup_{H\in\mc{H}}H,\tx{ for every }j\in[0,N]\setminus E.
		\]
	\end{itemize}
	Let \(\HC_G(E)\) and \(\PHC_G(E)\) denote the minimum sizes of a nontrivial hyperplane cover and a proper hyperplane cover respectively, for a weight-determined set \(\ul{E},\,E\subsetneq[0,N]\).  In the case \(G=\{0,1\}^n\), we will instead use the notations \(\HC_n(E)\) and \(\PHC_n(E)\) respectively.
	
	For any \(P(\mb{X})\in\mb{R}[\mb{X}]\), we denote by \(\mc{Z}_G(P)\), the set of all \(a\in G\) such that \(P(a)=0\).  For a weight-determined set \(\ul{E}\), where \(E\subsetneq[0,N]\), we say a polynomial \(P(\mb{X})\in\mb{R}[\mb{X}]\) is
	\begin{itemize}
		\item  a \tsf{nontrivial polynomial cover} of \(\ul{E}\) if
		\[
		\ul{E}\subseteq\mc{Z}_G(P)\ne G.
		\]
		\item  a \tsf{proper polynomial cover} of \(\ul{E}\) if
		\[
		\ul{E}\subseteq\mc{Z}_G(P)\quad\tx{and}\quad\ul{j}\not\subseteq\mc{Z}_G(P),\tx{ for every }j\in[0,N]\setminus E.
		\]
	\end{itemize}
	Let \(\PC_G(E)\) and \(\PPC_G(E)\) denote the minimum degree of a nontrivial polynomial cover and a proper polynomial cover respectively, for a weight-determined set \(\ul{E},\,E\subsetneq[0,N]\).  In the case \(G=\{0,1\}^n\), we will instead use the notations \(\PC_n(E)\) and \(\PPC_n(E)\).
	
	We are interested in the following problems.
	\begin{problem}\label{prob:main}
		Let \(G\) be a uniform grid.  For all \(E\subsetneq[0,N]\),
		\begin{enumerate}[(a)]
			\item  find \(\HC_G(E)\) and \(\PHC_G(E)\). 
			\item  find \(\PC_G(E)\) and \(\PPC_G(E)\).
		\end{enumerate}		
	\end{problem}
	All the above variants of hyperplane and polynomial covers coincide for the symmetric set \(\ul{[N]}\), and hence in our notation, Alon and F\"uredi~\cite{alon-furedi} proved the following.
	\begin{theorem}[\cite{alon-furedi}]\label{thm:alon-furedi}
		\(\HC_G([N])=\PHC_G([N])=\PC_G([N])=\PPC_G([N])=N\).
	\end{theorem}
	It is immediate from the definitions that \(\PHC_G(E)\ge\HC_G(E)\), \(\PPC_G(E)\ge\PC_G(E)\), \(\HC_G(E)\ge\PC_G(E)\), and \(\PPC_G(E)\ge\PHC_G(E)\), for all \(E\subsetneq[0,N]\).  In this work, we will give combinatorial characterizations of \(\PC_G(E)\) and \(\PPC_G(E)\), for all \(E\subsetneq[0,N]\), for an \SUt\,grid \(G\).  The characterization of \(\HC_G(E)\) and \(\PHC_G(E)\), for all \(E\subsetneq[0,N]\), for an \SUt\,grid \(G\) is left open.  However, in the Boolean cube setting, we will obtain \(\HC_n(E)=\PC_n(E)\) and \(\PHC_n(E)=\PPC_n(E)\), for all \(E\subsetneq[0,n]\).  In short, we will solve Problem~\ref{prob:main} (b) for \SUt\,grids, and further solve Problem~\ref{prob:main} (a) for the Boolean cube.  Problem~\ref{prob:main} (a) for \SUt\,grids is open.
	
	For every \(i\in[0,N]\), let \(T_{N,i}=[0,i-1]\cup[N-i+1,N]\). (This means \(T_{N,0}=\emptyset\).)  For any \(d,i\in[0,N],\,i\le d\) and \(E\subseteq[0,N]\), we define \(E\) to be \tsf{\((d,i)\)-admitting} if \(E\cup T_{N,i}\ne[0,N]\) and \(|E\setminus T_{N,i}|\le d-i\).  Further, we define \(E\) to be \tsf{\(d\)-admitting} if \(E\) is \((d,i)\)-admitting, for some \(i\in[0,d]\).  Our combinatorial characterizations that answer Problem~\ref{prob:main} (b) for \SUt\,grids are as follows.
	\begin{theorem}\label{thm:main-poly-cover}
		Let \(G\) be an \SUt\,grid.  For any \(E\subsetneq[0,N]\),
		\begin{enumerate}[(a)]
			\item  \(\PC_G(E)=\min\{d\in[0,N]:\tx{\(E\) is \(d\)-admitting}\}\).
			\item  \(\PPC_G(E)=|E|-\max\{i\in[0,\lfloor N/2\rfloor]:T_{N,i}\subseteq E\}\).
		\end{enumerate}
	\end{theorem}
	Further, we answer Problem~\ref{prob:main} (a), in the Boolean cube setting, as follows.
	\begin{theorem}\label{thm:main-hyp-cover}
		Consider the Boolean cube \(\{0,1\}^n\).  For any \(E\subsetneq[0,n]\),
		\begin{enumerate}[(a)]
			\item  \(\HC_n(E)=\PC_n(E)=\min\{d\in[0,n]:\tx{\(E\) is \(d\)-admitting}\}\).
			\item  \(\PHC_n(E)=\PPC_n(E)=|E|-\max\{i\in[0,\lfloor n/2\rfloor]:T_{n,i}\subseteq E\}\).
		\end{enumerate}
	\end{theorem}
	
	\subsection{Related work}\label{subsec:related}
	
	Alon and F\"uredi~\cite{alon-furedi} mention that their hyperplane covering problem was extracted by B\'ar\'any from Komj\'ath~\cite{komjath1994partitions}.  Some of the extensions and variants studied subsequently (over \(\mb{R}\)) are as follows.
	\begin{itemize}
		\item  Linial and Radhakrishnan~\cite{linial2005essential} gave an upper bound of \(\lceil n/2\rceil\) and a lower bound of \(\Omega(\sqrt{n})\) on the minimum size of \emph{essential hyperplane covers} of the Boolean cube -- a family of hyperplanes \(\mc{H}\) is an \tsf{essential hyperplane cover} if \(\mc{H}\) is a minimal family covering \(\{0,1\}^n\), and each coordinate is influential for a linear polynomial representing some hyperplane in \(\mc{H}\).  Saxton~\cite{saxton2013essential} later gave a tight bound of \(n+1\) for this problem, in the case where the linear polynomials representing the hyperplanes are restricted to be of the form \(\sum_{i\in[0,n]}a_iX_i-b\), where \(a_i\ge0\) for all \(i\in[n]\), and \(b\ge0\).  Recently, Yehuda and Yehudayoff~\cite{yehuda2021lower} improved the lower bound in the unrestricted case to \(\Omega(n^{0.52})\).
		
		\item  K\'os, M\'esz\'aros and R\'onyai~\cite{kos2012alon} introduced the following multiplicity extension: given a finite grid \(S=S_1\times\cdots\times S_n\) with each \((S_i,m_i)\) being a multiset such that \(0\in S_i,\,m_i(0)=1\), find the minimum number of hyperplanes such that each point \(s\in S\setminus\{0\}\) is covered by at least \(\sum_{i\in[n]}m_i(s_i)-n+1\) hyperplanes and the point 0 is not covered by any hyperplane.  They gave a tight lower bound of \(\sum_{i\in[n]}m_i(S_i)-n\).  This bound is in fact true over any field.
		
		\item  Aaronson, Groenland, Grzesik, Johnston and Kielak~\cite{aaronson2020exact} considered \emph{exact hyperplane covers} of subsets of the Boolean cube -- a family of hyperplanes \(\mc{H}\) is an \tsf{exact hyperplane cover} of a subset \(S\subseteq\{0,1\}^n\) if \(\big(\bigcap_{H\in\mc{H}}H\big)\cap\{0,1\}^n=S\).  They obtained tight bounds on the minimum size of exact hyperplane covers for subsets \(S\) with \(|\{0,1\}^n\setminus S|\le4\), and asymptotic bounds for general subsets.
		
		\item  Clifton and Huang~\cite{clifton2020almost} considered another multiplicity version of the hyperplane covering problem: find the least number of hyperplanes required to cover all points of the Boolean cube except the origin \(k\) times and not cover the origin at all.  They proved the tight bound of \(n+1\) and \(n+3\), for \(k=2\) and \(k=3\) respectively, and gave a lower bound of \(n+k+1\) for \(k\ge4\).  Sauermann and Wigderson~\cite{sauermann2020polynomials} considered the polynomial version of this problem: find the least degree of a polynomial that vanishes at all points of the Boolean cube, except the origin, \(k\) times and vanishes at the origin \(j\) times, for some \(j<k\).  They gave the tight bounds \(n+2k-3\) for \(j\le k-2\), and \(n+2k-2\) for \(j=k-1\).
	\end{itemize}
	
	Several more variants and extensions, in particular over positive characteristic, have appeared in the literature both before and after Alon and F\"uredi~\cite{alon-furedi} -- in Jamison~\cite{JAMISON1977253}, Brouwer~\cite{BROUWER1978251}, Ball~\cite{BALL2000441}, Zanella~\cite{ZANELLA2002381}, Ball and Serra~\cite{ball2009punctured}, Blokhuis~\cite{blokhuis2010covering}, and Bishnoi, Boyadzhiyska, Das and M\'esz\'aros~\cite{bishnoi2021subspace}, to quote a few.  For a detailed history of the hyperplane covering problems as well as the polynomial method, see, for instance, the nice introduction in~\cite{bishnoi2021subspace}.

	\subsection{Finite-degree Z-closures and Z*-closures, and polynomial covering problems}\label{subsec:comb-charac-Z*}
	
	The \emph{finite-degree Zariski closure} was defined by Nie and Wang~\cite{nie2015hilbert} towards obtaining a better understanding of the applications of the polynomial method to combinatorial geometry.  This is a closure operator\footnote{A closure operator on a set system \(\mc{F}\) (over a ground set) is any map \(\cl:\mc{F}\to\mc{F}\) satisfying:\n(i) \(A\subseteq\cl(A),\,\forall\,A\in\mc{F}\), \n(ii) \(\cl(A)\subseteq\cl(B),\,\forall\,A,B\in\mc{F},\,A\subseteq B\), and \n(iii) \(\cl(\cl(A))=\cl(A),\,\forall\,A\in\mc{F}\).  This is a well-studied set operator.  See for instance, Birkhoff~\cite[Chapter V, Section 1]{birkhoff1973lattice} for an introduction.} and has been studied implicitly even earlier (see for instance, Wei~\cite{wei1991}, Heijnen and Pellikaan~\cite{heijnen-pellikaan}, Keevash and Sudakov~\cite{keevash-sudakov-inclusion}, and Ben-Eliezer, Hod and Lovett~\cite{BHL}).  We will abbreviate the term `Zariski closure' by `Z-closure'.
	
	\subsubsection{Finite-degree Z-closures and Z*-closures}
	
	Let \(G\) be a uniform grid.  For any \(d\in[0,N]\) and any \(S\subseteq G\), the \tsf{degree-\(d\) Z-closure} of \(S\), denoted by \(\zcl_{G,d}(S)\), is defined to be the common zero set, in \(G\), of all polynomials that vanish on \(S\), and have degree at most \(d\).\footnote{The set of all polynomials that vanish on \(S\) and have degree at most \(d\) is a vector space over \(\mb{R}\).  Note that here we include the zero polynomial in the set.  To facilitate this, we adopt the convention that the degree of the zero polynomial is \(-\infty\).}  In the case \(G=\{0,1\}^n\), we will instead use the notation \(\zcl_{n,d}(S)\).
	
	The finite-degree Z-closures are relevant to us in the context of our polynomial covering problems.  We are interested in polynomial covering problems that impose conditions on weight-determined sets, and thus, we are interested in finite-degree Z-closures of weight-determined sets.  Intriguingly, the finite-degree Z-closure of a weight-determined set need not be weight-determined.
	
	For instance, consider the \SUt\,grid \(G=[0,2]^3\).  In this case, we have \(N=6\) and \(T_{6,3}=[0,2]\cup[4,6]\).  We have \(\ul{3}=\{(2,1,0),(1,2,0),(0,2,1),(0,1,2),(2,0,1),(1,0,2),(1,1,1)\}\).  Consider
	\[
	P(X_1,X_2,X_3)=X_1X_2(X_1-X_2)+X_2X_3(X_2-X_3)-X_1X_3(X_1-X_3).
	\]
	Clearly \(\deg P=3\).  It is easy to check that \(P|_{\ul{T_{6,3}}}=0\).  Further, we get
	\[
	P(2,1,0)=P(0,2,1)=P(1,0,2)=2\quad\tx{and}\quad P(1,2,0)=P(0,1,2)=P(2,0,1)=-2.
	\]
	So \(a\not\in\zcl_{G,3}(\ul{T_{6,3}})\), for all \(a\in\ul{3},\,a\ne(1,1,1)\).  Further, consider any \(Q(X_1,X_2,X_3)\in\mb{R}[X_1,X_2,X_3]\) such that \(Q|_{\ul{T_{6,3}}}=0\) and \(\deg Q\le 3\).  Let
	\[
	R(X_1,X_2,X_3)=Q(X_1,X_2,X_3)(X_1^2+X_2^2+X_3^2-5).
	\]
	So \(\deg R=5\).  Then we have \(R|_{\ul{T_{6,3}}}=0\), and \(R(a)=0\), for all \(a\in\ul{3},\,a\ne(1,1,1)\).  Thus \(R(x)=0\), for all \(x\in G,\,x\ne(1,1,1)\).  If \(R(1,1,1)\ne0\), then by Theorem~\ref{thm:alon-furedi-basic}, we have \(\deg R\ge 6\), which is not true.  So \(R(1,1,1)=0\), which implies \(Q(1,1,1)=0\).  Thus \((1,1,1)\in\zcl_{G,3}(\ul{T_{6,3}})\).  Hence we have \(\ul{3}\cap\zcl_{G,3}(\ul{T_{6,3}})\ne\emptyset\) but \(\ul{3}\not\subseteq\zcl_{G,3}(\ul{T_{6,3}})\), which implies \(\zcl_{G,3}(\ul{T_{6,3}})\) is not weight-determined.
	
	We will circumvent this issue by introducing a new closure operator, defined exclusively for weight-determined sets.  Let \(G\) be a uniform grid.  For any \(d\in[0,N]\) and \(E\subseteq[0,N]\), we define the \tsf{degree-\(d\) Z*-closure} of \(\ul{E}\), denoted by \(\zscl_{G,d}(\ul{E})\), to be the maximal weight-determined set contained in \(\zcl_{G,d}(\ul{E})\).  In other words, \(\zcl_{G,d}(\ul{E})\) is defined by the implications:
	\begin{align*}
		\ul{j}\subseteq\zcl_{G,d}(\ul{E})&\implies\ul{j}\subseteq\zscl_{G,d}(\ul{E}),\\
		\tx{and}\quad\ul{j}\not\subseteq\zcl_{G,d}(\ul{E})&\implies\ul{j}\cap\zscl_{G,d}(\ul{E})=\emptyset.
	\end{align*}
	It follows easily that \(\zscl_{G,d}\) is a closure operator\footnote{\(\zscl_{G,d}\) is a closure operator on the family of all weight-determined sets of \(G\).}; we will prove this in Section~\ref{sec:prelims}.
	
	\vspace{-0.5cm}
	\paragraph*{Notation.}  By definition, it is clear that the finite-degree Z*-closure is a weight-determined set.  So, whenever convenient, we will use our identification of weight-determined sets with subsets of \([0,N]\) while describing these closures.  For \(E\subseteq[0,N]\), the notation \(\zscl_{G,d}(\ul{E})\) will denote the Z*-closure as a subset of \(G\), and the notation \(\zscl_{G,d}(E)\) will denote the Z*-closure as a subset of \([0,N]\).  Similar `double notations' would also apply to Z-closures of symmetric sets of the Boolean cube.
	
	The relevance of the finite-degree Z*-closures to our polynomial covering problems is captured by the following simple lemma, which is quite immediate from the definitions.
	\begin{lemma}\label{lem:poly-cover-defn}
		Let \(G\) be a uniform grid.  For any \(E\subsetneq[0,N]\),
		\begin{enumerate}[(a)]
			\item  \(\PC_G(E)=\min\{d\in[0,N]:\zscl_{G,d}(E)\ne[0,N]\}\).
			\item  \(\PPC_G(E)=\min\{d\in[0,N]:\zscl_{G,d}(E)=E\}\).
		\end{enumerate}
	\end{lemma}

	\subsubsection{Combinatorial characterization of finite-degree Z*-closures}
	
	We will proceed to give a combinatorial characterization of finite-degree Z*-closures.  We need a couple of set operators to proceed.
	
	\paragraph{Two set operators.}  Fix \(M\in\mb{Z}^+\) and for every \(d\in[0,M]\), define a set operator \(L_{M,d}:2^{[0,M]}\to2^{[0,M]}\) as follows.  Let \(E=\{t_1<\cdots<t_s\}\subseteq[0,M]\).  Define
	\[
	L_{M,d}(E)=\begin{cases}
		E&\tx{if }|E|\le d,\\
		[0,t_{s-d}]\cup E\cup[t_{d+1},M]&\tx{if }|E|\ge d+1.
	\end{cases}
	\]
	We are interested in iterated applications of the operator \(L_{M,d}\), and in an obvious way, for any \(k\in\mb{N}\), we define the operator \(L_{M,d}^{k+1}=L_{M,d}\circ L_{M,d}^k\), where \(L_{M,d}^0\) denotes the identity operator.  Clearly, for every \(d\in[0,M]\) and \(E\subseteq[0,M]\), we have the chain
	\[
	E=L_{M,d}^0(E)\subseteq L_{M,d}(E)\subseteq L_{M,d}^2(E)\subseteq\cdots
	\]
	So for every \(d\in[0,M]\), define the set operator \(\ol{L}_{M,d}:2^{[0,M]}\to2^{[0,M]}\) as
	\[
	\ol{L}_{M,d}(E)=\bigcup_{k\ge0}L_{M,d}^k(E),\quad\tx{for all }E\subseteq[0,M].
	\]
	
	We are now ready to state our main theorem: our combinatorial characterization of finite-degree Z*-closures of weight-determined sets of an \SUt\,grid.
	\begin{theorem}\label{thm:main-poly-closure-op}
		Let \(G\) be an \SUt\,grid.  For every \(d\in[0,N]\) and \(E\subseteq[0,N]\),
		\[
		\zscl_{G,d}(E)=\ol{L}_{N,d}(E).
		\]
	\end{theorem}
	In addition, we can characterize when the finite-degree Z*-closure of a weight-determined set \(E\) is equal to \(E\) or \([0,N]\).
	\begin{proposition}\label{prop:poly-L-finer}
		Let \(G\) be an \SUt\,grid.  For every \(d\in[0,N]\) and \(E\subseteq[0,N]\),
		\begin{enumerate}[(a)]
			\item  \(\ol{L}_{N,d}(E)\ne[0,N]\) if and only if \(E\) is \(d\)-admitting.
			\item  \(\ol{L}_{N,d}(E)=E\) if and only if \(T_{N,|E|-d}\subseteq E\).
		\end{enumerate}
	\end{proposition}
	
	This would complete the proof of Theorem~\ref{thm:main-poly-cover}, thus completing our solution to Problem~\ref{prob:main} (b) for \SUt\,grids.

	\subsection{Finite-degree h-closures and h*-closures, and hyperplane covering problems}\label{subsec:hyp-clo}
	
	To better understand the difference between the hyperplane and polynomial covering problems, we introduce another new closure operator, which we call the \emph{finite-degree {\normalfont h}-closure}, defined using polynomials representing hyperplane covers.  Let \(\ms{H}_n\) be the set of all polynomials in \(\mb{R}[\mb{X}]\coloneqq\mb{R}[X_1,\ldots,X_n]\) which are products of polynomials of degree at most 1.  Let \(G\) be a uniform grid.  For any \(d\in[0,N]\) and any \(S\subseteq G\), we define the \tsf{degree-\(d\) h-closure} of \(S\), denoted by \(\hcl_{G,d}(S)\), to be the common zero set, in \(G\), of all polynomials in \(\ms{H}_n\) that vanish on \(S\) and have degree at most \(d\).  By definition, it is clear that \(\zcl_{G,d}(S)\subseteq\hcl_{G,d}(S)\).  In the case \(G=\{0,1\}^n\), we will instead use the notation \(\hcl_{n,d}(S)\).
	
	Note that we do not know if the finite-degree hyperplane closure of every weight-determined set of \(G\) is weight-determined.  So, akin to the definition of finite-degree Z*-closures, for any \(d\in[0,N]\) and \(E\subseteq[0,N]\), we define the \tsf{degree-\(d\) h*-closure} of \(\ul{E}\), denoted by \(\hscl_{G,d}(\ul{E})\), to be the maximal weight-determined set contained in \(\hcl_{G,d}(\ul{E})\).  These closures are relevant to us in the context of our hyperplane covering problems due to the following observation, which is immediate from the definitions.
	\begin{observation}\label{obs:hyp-cover-defn}
		Let \(G\) be a uniform grid.  For any \(E\subsetneq[0,N]\),
		\begin{enumerate}[(a)]
			\item  \(\HC_G(E)=\min\{d\in[0,N]:\hscl_{G,d}(\ul{E})\ne G\}\).
			\item  \(\PHC_G(E)\ge\min\{d\in[0,N]:\hscl_{G,d}(\ul{E})=\ul{E}\}\).
		\end{enumerate}
	\end{observation}
	
	\subsubsection{The Boolean cube setting:  characterizing h-closures}
	
	We will characterize the finite-degree h-closures of all symmetric sets of the Boolean cube for all degrees; in fact, we make the intriguing observation that these coincide with the finite-degree Z-closures.
	
	\vspace{-0.5cm}
	\paragraph*{Notation.}  It is easy to see that the finite-degree h-closure of a symmetric set of a Boolean cube is symmetric.  So \(\hcl_{n,d}(\ul{E})=\hscl_{n,d}(\ul{E})\), for all \(E\subseteq[0,n],\,d\in[0,n]\).  So once again, we will use the identification between symmetric sets of \(\{0,1\}^n\) and subsets of \([0,n]\).  For \(E\subseteq[0,n]\), the notation \(\hcl_{n,d}(\ul{E})\) will denote the h-closure as a subset of \(\{0,1\}^n\), and the notation \(\hcl_{n,d}(E)\) will denote the h-closure as a subset of \([0,n]\).  We will also follow the same convention if we use the notation of h*-closures.
	
	We already have Theorem~\ref{thm:main-poly-closure-op} that characterizes the finite-degree Z*-closures. Further, since the finite-degree Z-closure of a symmetric set is symmetric, we have \(\zcl_{n,d}(\ul{E})=\zscl_{n,d}(\ul{E})\), for all \(E\subseteq[0,n],\,d\in[0,n]\).  With an additional observation, we will conclude the following.
	\begin{theorem}\label{thm:hcl=zcl}
		For every \(d\in[0,n]\) and \(E\subseteq[0,n]\),
		\[
		\hcl_{n,d}(E)=\hscl_{n,d}(E)=\zcl_{n,d}(E)=\zscl_{n,d}(E)=\ol{L}_{n,d}(E).
		\]
	\end{theorem}
	Further, using Observation~\ref{obs:hyp-cover-defn}, Theorem~\ref{thm:hcl=zcl} and a tight construction of hyperplane cover, we will prove Theorem~\ref{thm:main-hyp-cover}, our solution to Problem~\ref{prob:main} (a) in the Boolean cube setting.
	
	It must be noted that for larger uniform grids, for a weight-determined set, the finite-degree h-closure and Z-closure need not be equal.  For instance, let \(G=[0,2]\times[0,2]\) and consider \(T_{4,2}=[0,1]\cup[3,4]\).  Owing to the fact that affine hyperplanes in \(\mb{R}^2\) are lines, we get \(\hcl_{G,2}(\ul{T_{4,2}})=G\).  Further, let \(P(X_1,X_2)=X_1^2-X_1X_2+X_2^2-X_1-X_2\).  Then obviously \(\deg P=2\), and we can check that \(P|_{\ul{T_{4,2}}}=0\).  Note that \(G\setminus\ul{T_{4,2}}=\ul{2}=\{(2,0),(1,1),(0,2)\}\).  We have \(P(2,0)=2,\,P(1,1)=-1,\,P(0,2)=2\).  Thus \(\ul{2}\cap\zcl_{G,2}(\ul{T_{4,2}})=\emptyset\), that is, \(\zcl_{G,2}(\ul{T_{4,2}})=\ul{T_{4,2}}\).
	
	We do not yet know how to approach the finite-degree hyperplane closures for larger uniform grids.  We, therefore, have the following open questions.
	\begin{openproblem}\label{open:hyp-cover}
		Let \(G\ne\{0,1\}^n\) be a uniform (or \SUt) grid.
		\begin{enumerate}[(a)]
			\item  Characterize \(\hcl_{G,d}(\ul{E})\) and \(\hscl_{n,d}(\ul{E})\), for all \(E\subseteq[0,N]\).
			
			\item  Solve Problem~\ref{prob:main} (a), that is, determine \(\HC_G(E)\) and \(\PHC_G(E)\), for all \(E\subsetneq[0,N]\).
		\end{enumerate}
	\end{openproblem}
	
	\subsubsection{The Boolean cube setting:  connection with the affine Hilbert function} 
	
	Now for any subset \(A\subseteq\mb{R}^n\), let \(V(A)\) denote the vector space of all functions \(A\to\mb{R}\).  For \(d\ge0\), let \(V_d(A)\) denote the subspace of all functions that admit a polynomial representation with degree at most \(d\).  The \tsf{affine Hilbert function} of \(A\) is defined by \(\Hilbert_d(A)=\dim V_d(A),\,d\ge0\).  This is a well-studied object in the literature.  (See for instance, Cox, Little and O'Shea~\cite[Chapter 9, Section 3]{cox2015ideals} for an introduction.)
	
	Let us fix some notations. Consider the Boolean cube \(\{0,1\}^n\).  Let \(d\in[0,n]\).  For any \(E\subseteq[0,n]\), let \(r_d(E)=|E\setminus[0,d]|,\,\ell_d(E)=|[0,d]\setminus E|\).  Further, denote the enumerations \(E\setminus[0,d]=\{j^+_1<\cdots<j^+_{r_d(E)}\}\) and \([0,d]\setminus E=\{j^-_{\ell_d(E)}<\cdots<j^-_1\}\).  Bernasconi and Egidi~\cite{bernasconi-egidi-hilbert} characterized the affine Hilbert functions of all symmetric sets of the Boolean cube.
	\begin{theorem}[\cite{bernasconi-egidi-hilbert}]\label{thm:BE-hilbert}
		Consider the Boolean cube \(\{0,1\}^n\).  For any \(d\in[0,n]\) and \(E\subseteq[0,n]\),
		\[
		\Hilbert_d(E)=\sum_{j\in E\cap[0,d]}\binom{n}{j}+\sum_{t=1}^{\min\{r_d(E),\,\ell_d(E)\}}\min\bigg\{\binom{n}{j^+_t},\binom{n}{j^-_t}\bigg\}.
		\]
	\end{theorem}
	
	The following fact is folklore and follows easily from the definitions.  (See, for instance, Nie and Wang~\cite[Proposition 5.2]{nie2015hilbert} for a proof.)  It connects affine Hilbert functions with finite-degree Z-closures.
	\begin{fact}[Folklore,~\cite{nie2015hilbert}]\label{thm:H-NW}
		Let \(d\in[0,n]\) and \(A\subseteq\{0,1\}^n\).  If \(A\subseteq B\subseteq\zcl_{n,d}(A)\), then \(\Hilbert_d(A)=\Hilbert_d(B)\).  In particular, \(\Hilbert_d(A)=\Hilbert_d(\zcl_{n,d}(A))\).
	\end{fact}
	We remark that using Fact~\ref{thm:H-NW} and the result of Bernasconi and Egidi (Theorem~\ref{thm:BE-hilbert}), we could obtain Theorem~\ref{thm:hcl=zcl}, that is, our combinatorial characterization of finite-degree Z*-closures (as well as finite-degree Z-closures and h-closures) of symmetric sets of the Boolean cube.  However, our arguments, in fact, prove Theorem~\ref{thm:main-poly-closure-op}, that is, our proof works over general uniform grids and could also be regarded as being more combinatorial.
	
	\paragraph*{Linear time algorithm.}  We give a linear time algorithm (Algorithm~\ref{algo:main}) to compute \(\ol{L}_{N,d}(E)\), for any \(d\in[0,N],\,E\subseteq[0,N]\).  Thus, the finite-degree Z*-closures (Z-closures in the Boolean cube setting) can be computed in linear time using our characterization.  However, it is unclear if the finite-degree Z-closures can be computed in linear time, in the Boolean cube setting, using Fact~\ref{thm:H-NW} and the result of Bernasconi and Egidi (Theorem~\ref{thm:BE-hilbert}).

	\subsection{Other applications}\label{subsec:applications}
	
	We believe that the combinatorial characterization in Theorem~\ref{thm:main-poly-closure-op} might also be of independent interest.  Indeed, we give a couple of other applications.
	
	\subsubsection{An alternate proof of a lemma by Alon et al.~\cite{alon1988balancing}}\label{subsubsec:Alon-lemma}
	
	Alon, Bergmann, Coppersmith and Odlyzko~\cite{alon1988balancing} obtained a tight lower bound for a \emph{balancing problem} on the Boolean cube \(\{-1,1\}^n\).  Their proof is via the polynomial method, using the following lemma.
	
	
	\begin{lemma}[\cite{alon1988balancing}]\label{lem:alon-lemma}
		Let \(n\in\mb{Z}^+\) be even and \(f(\mb{X})\in\mb{R}[\mb{X}]\) represent a nonzero function on \(\{-1,1\}^n\) such that either of the following conditions is true.
		\begin{itemize}
			\item  \(f(x)=0\), for all \(x\in\{-1,1\}^n\) having an even number of \(i\in[n]\) such that \(x_i=-1\).
			\item  \(f(x)=0\), for all \(x\in\{-1,1\}^n\) having an odd number of \(i\in[n]\) such that \(x_i=-1\).
		\end{itemize}
		Then \(\deg f\ge n/2\).
	\end{lemma}
	For \(i\in\{0,1\}\), let \(E_i=\{j\in[0,n]:j\equiv i\modulo{2}\}\).  We will observe that the above lemma is equivalent to the following proposition, thus giving us an alternate proof.
	\begin{proposition}\label{prop:ABCO-equivalent}
		If \(n\in\mb{Z}^+\) is even, then \(\ol{L}_{n,n/2-1}(E_0)=\ol{L}_{n,n/2-1}(E_1)=[0,n]\).
	\end{proposition}
	
	\subsubsection{Certifying degrees of weight-determined sets}\label{subsubsec:cert-deg}
	
	Let \(G\) be a uniform grid.  We say a polynomial \(f(\mb{X})\in\mb{R}[\mb{X}]\) is a \tsf{certifying polynomial} for a subset \(S\subseteq G\) if \(f\in V(G)\) is nonconstant, but \(f|_S\) is constant.  (Thus, \(S\) is necessarily a proper subset.)  The notion of a certifying polynomial (in the Boolean cube setting, for Boolean functions and Boolean circuits) was studied by Kopparty and Srinivasan~\cite{kopparty2018certifying} to prove lower bounds for a certain class of Boolean circuits.  Variants of this notion have appeared in theoretical computer science, specifically in complexity theory literature, in the works of Aspnes, Beigel, Furst and Rudich~\cite{aspnes93theexpressive}, Green~\cite{Green_complexfourier}, and Alekhinovich and Razborov~\cite{alekh-razborov-2001}.  Certifying polynomials have also appeared in the context of cryptography in Carlet, Dalai, Gupta and Maitra~\cite{carlet2006}.
	
	For a subset \(S\subseteq G\), the \tsf{certifying degree} of \(S\), denoted by \(\certdeg(S)\), is defined to be the smallest \(d\in[0,N]\) such that \(S\) has a certifying polynomial with degree at most \(d\).  We will determine the certifying degrees of all weight-determined sets, when \(G\) is an \SUt\,grid.  In fact, we will get the following.
	\begin{theorem}\label{thm:certdeg}
		Let \(G\) be an \SUt\,grid.  For any proper weight-determined set \(E\subsetneq[0,N]\),
		\[
		\certdeg(E)=\PC_G(E)=\min\{d\in[0,N]:E\tx{ is \(d\)-admitting}\}.
		\]
	\end{theorem}
	
	We will conclude our work by considering a third variant of our covering problems -- the \emph{exact covering problem}.  As for the other covering problems, we will define (in Section~\ref{sec:exact-cover}) an \emph{exact hyperplane (polynomial) cover} for a weight-determined set \(\ul{E},\,E\subsetneq[0,N]\) in a uniform grid \(G\), and consider the minimum size (degree) of an exact hyperplane (polynomial) cover for \(\ul{E}\).  Let us denote these by \(\EHC_G(E)\) and \(\EPC_G(E)\) respectively.  (In the Boolean cube setting, we will denote these by \(\EHC_n(E)\) and \(\EPC_n(E)\) respectively.)  We will see some partial results and conjecture a solution in the Boolean cube setting.
	
	We can quickly indicate the progress made in our work via the following tables.
	
	\begin{table}[h]
		\centering
		\bgroup
		\def\arraystretch{1.5}
		\begin{tabular}{|c|c|c|c|}
			\hline
			\tbf{Cover}&\tbf{Nontrivial}&\tbf{Proper}&\tbf{Exact}\\
			\hline
			\tbf{Hyperplane}&\cellcolor{SpringGreen}\(\HC_n\)&\cellcolor{Yellow}\(\PHC_n\)&\(\EHC_n\)\\
			\hline
			\tbf{Polynomial}&\cellcolor{SpringGreen}\(\PC_n\)&\cellcolor{Yellow}\(\PPC_n\)&\cellcolor{Yellow}\(\EPC_n\)\\
			\hline
		\end{tabular}
		\egroup
		\caption{\label{tab:cube}For the Boolean cube \(\{0,1\}^n\)}
	\end{table}

	\begin{table}[h]
		\centering
		\bgroup
		\def\arraystretch{1.5}
		\begin{tabular}{|c|c|c|c|}
			\hline
			\tbf{Cover}&\tbf{Nontrivial}&\tbf{Proper}&\tbf{Exact}\\
			\hline
			\tbf{Hyperplane}&\(\HC_G\)&\(\PHC_G\)&\(\EHC_G\)\\
			\hline
			\tbf{Polynomial}&\cellcolor{SpringGreen}\(\PC_G\)&\cellcolor{Yellow}\(\PPC_G\)&\(\EPC_G\)\\
			\hline
		\end{tabular}
		\egroup
		\caption{\label{tab:grid}For an \SUt\,grid \(G\)}
	\end{table}
	In this work, we have obtained combinatorial characterizations of the quantities in the colored cells in the above tables.  Further, the cells with the same color have the same characterization, that is, although the corresponding questions are different, the answers are the same.  Characterizations of the quantities in the uncolored (white) cells is open.

	\paragraph{Digression 1:  Why consider only uniform grids?}  It is easy to see that all the results in this work that hold for a uniform grid \(G\) would also hold for the image of \(G\) under any invertible affine transformation of \(\mb{R}^n\).  A typical grid that is genuinely nonuniform would be of the form \(H=\{t^{(1)}_0<\cdots<t^{(1)}_{k_1-1}\}\times\cdots\times\{t^{(n)}_0<\cdots<t^{(n)}_{k_n-1}\}\), where the differences \(t^{(j)}_{i+1}-t^{(j)}_i,\,i\in[0,k_j-2]\) are not equal, for some \(j\in[n]\).  In this case, an obvious way to define the weight of an element \(x\in H\) would be \(\wgt(x)\coloneqq\sum_{j\in[n]}s_j\), where \(x_j=t^{(j)}_{s_j}\) for all \(j\in[n]\).\footnote{This definition of weight is very natural, if one appeals to lattice theory.  Indeed, \(H\) is a lattice w.r.t. a suitable partial order with \(\wgt\), as defined, being the rank function of the lattice.}  Unlike in the uniform setting, it is no longer true that every weight-determined set \(\ul{w},\,w\in[0,N]\) is contained in the zero set of the linear polynomial \(\sum_{j\in[n]}X_j-w\).  This has undesirable ripple effects; for instance, the finite-degree Z*-closures of weight-determined sets depend on the weights of the points, as well as on the coordinates of the points in the set.
	
	As an example, let \(G=\{0,1,2\}^2\) and \(H=\{0,1,3\}^2\).  Both grids have the same dimensions and \(N=4\).  In \(G\), we have \(\ul{2}=\{(2,0),(1,1),(0,2)\}\), and further \(\mc{Z}_G(X_1+X_2-2)=\ul{2}\).  So \(\zscl_{G,1}(\{2\})=\{2\}\).  In \(H\), we have \(\ul{2}=\{(3,0),(1,1),(0,3)\}\).  If \(\ul{2}\subseteq\mc{Z}_H(P)\), for \(P(X_1,X_2)=aX_1+bX_2+c,\n a,b,c\in\mb{R}\), then we get the linear equations
	\[
	3a+c=0,\quad a+b=0,\quad 3b+c=0,
	\]
	which implies \(a=b=c=0\), that is, \(P(X_1,X_2)=0\).  So \(\zscl_{H,1}(\{2\})=[0,4]\).
	
	If the grid is uniform, however, the finite-degree Z*-closures of weight-determined sets depend only on the weights of the points in the set. So the upshot is that the `uniform' condition on grids ensures the setting is nice enough for our tools, tricks and techniques to work well, and give neat results.
	
	\paragraph{Digression 2: What happens over fields other than \(\mb{R}\)?}  Consider the Boolean cube \(\{0,1\}^n\), as well as any uniform grid \(G\subseteq\mb{N}^n\).  A crucial observation that enables our algebraic arguments (over \(\{0,1\}^n\) as well as over \(G\)) to go through is that for any valid \(w\in\mb{N}\), the zero set of the linear polynomial \(\sum_{i\in[n]}X_i-w\) is exactly \(\ul{w}\) (in \(\{0,1\}^n\) as well as in \(G\)).  This is true over any field of characteristic zero.  Indeed, all the results in this paper hold true over any field of characteristic zero, without any change in the arguments.  Further, it is easy to see that the above mentioned property is also true, and therefore all the results in this paper also hold true, over any field with large positive characteristic \(p\):  we require \(p>n\) over \(\{0,1\}^n\), and \(p>N\) over \(G\).  However, our results do not extend to fields with small positive characteristic.
	
	\paragraph{Organization of the paper.}  In Section~\ref{sec:prelims}, we will look at some preliminaries concerning the different closure operators of interest to us.  In Section~\ref{sec:Z-star-clo}, we will prove our combinatorial characterization of Z*-closures of weight-determined sets of an \SUt\,grid, and then obtain Theorem~\ref{thm:main-poly-cover}, our more refined solution to Problem~\ref{prob:main} (b) for \SUt\,grids.  In Section~\ref{sec:h-clo-hyper}, we will see that the h-closures coincide with the Z-closures and Z*-closures, for all symmetric sets of the Boolean cube.  We also obtain Theorem~\ref{thm:main-hyp-cover}, our solution to Problem~\ref{prob:main} (a) in the Boolean cube setting.  In Section~\ref{sec:further-appl}, we consider further applications of our combinatorial characterization of finite-degree Z*-closures from Section~\ref{sec:Z-star-clo}:\n(i) an alternate proof of a lemma by Alon et al.~(1988) in the context of \emph{balancing problems}, and\n(ii) a characterization of the \emph{certifying degrees} of weight-determined sets.  We conclude in Section~\ref{sec:exact-cover} by introducing a third variant of our covering problems.  We will discuss some partial results and conjecture a solution in the Boolean cube setting.

	\section{Preliminaries}\label{sec:prelims}
	
	We begin by recalling some definitions.  Consider the uniform grid \(G=[0,k_1-1]\times\cdots\times[0,k_n-1]\) with \(N=\sum_{i\in[n]}(k_i-1)\).  We will require the following important result by Alon~\cite{alon_1999}.  The result of Alon and F\"uredi (Theorem~\ref{thm:alon-furedi-basic}, Theorem~\ref{thm:alon-furedi}), in fact, follows from this.
	\begin{theorem}[Combinatorial Nullstellensatz~{\cite[Theorem 2]{alon_1999}}]\label{thm:CN}
		If \(P(\mb{X})\in\spn\{\mb{X}^\gamma:\gamma\in G\}\) is a nonzero polynomial, then there exists \(a\in G\) such that \(P(a)\ne0\).
	\end{theorem}
	
	We adopt the convention that \(\deg(0)=-\infty\), where 0 denotes the zero polynomial.  For a subset \(I\subseteq\mb{R}[\mb{X}]\), let \(\mc{Z}(I)=\{a\in G:f(a)=0,\tx{ for all }f(\mb{X})\in I\}\).  Further, for any \(d\in[0,N]\) and \(I\subseteq\mb{R}[\mb{X}]\), we denote \(I_d=\{f(\mb{X})\in I:\deg f\le d\}\).  Also, for any subset \(S\subseteq G\), let \(\mc{I}(I,S)=\{f(\mb{X})\in I:f(a)=0,\tx{ for all }a\in S\}\), and denote \(\mc{I}(S)=\mc{I}(\mb{R}[\mb{X}],S)\).
	
	For any \(d\in[0,N]\) and \(S\subseteq G\), the \tsf{degree-\(d\) Z-closure} of \(S\) is defined as \(\zcl_{G,d}(S)=\mc{Z}(\mc{I}(\mb{R}[\mb{X}]_d,S))\).
	
	When \(G=\{0,1\}^n\), it follows that the finite-degree Z-closure of a symmetric set is symmetric.  However, for a general uniform grid \(G\), the finite-degree Z-closure of a weight-determined set need not be weight-determined.  We recall the one-to-one correspondence between weight-determined sets of \(G\) and subsets of \([0,N]\).  Every subset \(E\subseteq[0,N]\) corresponds to the weight-determined set \(\ul{E}=\{x\in G:\wgt(x)\in E\}\).  For convenience, we will identify \(E\) with \(\ul{E}\).  We therefore consider a new notion of closure exclusiely for weight-determined sets.  Let \(G\) be a uniform grid.  For \(d\in[0,N]\) and \(E\subseteq[0,N]\), we define the \tsf{degree-\(d\) Z*-closure} of \(\ul{E}\) as
	\[
	\zscl_{G,d}(\ul{E})=\bigcup_{\substack{j\in[0,N]\\\ul{j}\subseteq\zcl_{G,d}(\ul{E})}}\ul{j}.
	\]
	The following properties of finite-degree Z*-closures are similar to that of finite-degree Z-closures, and follow quickly from the definition.
	\begin{proposition}\label{pro:Z*-prop}
		Let \(G\) be a uniform grid and \(d\in[0,N]\).
		\begin{enumerate}[(a)]
			\item  \(\zscl_{G,d}(\ul{E})\) is weight-determined, for all \(E\subseteq[0,N]\).
			\item  \(\zscl_{G,d}\) is a closure operator.
			\item  \(\zscl_{G,d+1}(\ul{E})\subseteq\zscl_{G,d}(\ul{E})\), for all \(E\subseteq[0,N]\).
		\end{enumerate}
	\end{proposition}
	\begin{proof}
		\begin{enumerate}[(a)]
			\item  By definition, if \(x\in\zscl_{G,d}(\ul{E}),\,x\in\ul{j}\), then \(\ul{j}\subseteq\zscl_{G,d}(\ul{E})\).  So \(\zscl_{G,d}(\ul{E})\) is weight-determined.
			\item  Note that, as a set operator, we have \(\zscl_{G,d}:\tsf{W}(G)\to\tsf{W}(G)\), where \(\tsf{W}(G)\) denotes the collection of all weight-determined sets of \(G\).  We observe the following.
			\begin{enumerate}[(i)]
				\item  Clearly \(\ul{E}\subseteq\zscl_{G,d}(\ul{E})\), for all \(E\subseteq[0,N]\).
				\item  Now consider any \(A\subseteq B\subseteq[0,N]\).  We already have \(\zcl_{G,d}(\ul{A})\subseteq\zcl_{G,d}(\ul{B})\).  So for \(j\in[0,N]\), if \(\ul{j}\subseteq\zcl_{G,d}(\ul{A})\), then obviously \(\ul{j}\subseteq\zcl_{G,d}(\ul{B})\).  Thus \(\zscl_{G,d}(\ul{A})\subseteq\zscl_{G,d}(\ul{B})\).
				\item  Further, for any \(E\subseteq[0,N]\), since \(\zcl_{G,d}(\zcl_{G,d}(\ul{E}))=\zcl_{G,d}(\ul{E})\), we have in fact \(\zscl_{G,d}(\zcl_{G,d}(\ul{E}))=\zscl_{G,d}(\ul{E})\).  So we have, using (i) and (ii),
				\[
				\zscl_{G,d}(\ul{E})\subseteq\zscl_{G,d}(\zscl_{G,d}(\ul{E}))\subseteq\zscl_{G,d}(\zcl_{G,d}(\ul{E}))=\zscl_{G,d}(\ul{E}).
				\]
			\end{enumerate}
			Thus \(\zscl_{G,d}\) is a closure operator.
			\item  Consider any \(E\subseteq[0,N]\).  We already have \(\zcl_{G,d+1}(\ul{E})\subseteq\zcl_{G,d}(\ul{E})\).  So for \(j\in[0,N]\), if \(\ul{j}\subseteq\zcl_{G,d+1}(\ul{E})\), then obviously \(\ul{j}\subseteq\zcl_{G,d}(\ul{E})\).  Thus \(\zscl_{G,d+1}(\ul{E})\subseteq\zscl_{G,d}(\ul{E})\).\qedhere
		\end{enumerate}
	\end{proof}
	
	
	For our results, we are most interested in the setting where the uniform grid \(G\) is \tsf{strictly unimodal} (\SUt), that is, we have
	\[
	\sqbinom{G}{0}<\cdots<\sqbinom{G}{\lfloor N/2\rfloor}=\sqbinom{G}{\lceil N/2\rceil}>\cdots>\sqbinom{G}{N}.
	\]
	There is a simple characterization of \SUt\,grids given by Dhand~\cite{DHAND201420}.
	\begin{theorem}[\cite{DHAND201420}]\label{thm:strict-unimodal}
		A uniform grid \(G\) is strictly unimodal if and only if
		\[
		2\max_{i\in[n]}\,(k_i-1)\le\sum_{i\in[n]}(k_i-1)+1.
		\]
	\end{theorem}
	
	We will also need the following abbreviation. For any \(a_1\in\{0,1\}^{n_1},\,a_2\in\{0,1\}^{n_2},\ldots,\,a_k\in\{0,1\}^{n_k}\), the vector
	\[
	\big(\underbrace{a_1,\ldots,a_1}_{i_1},\,\underbrace{a_2,\ldots,a_2}_{i_2},\ldots,\,\underbrace{a_k,\ldots,a_k}_{i_k}\big)\in\{0,1\}^{n_1i_1+n_2i_2+\cdots+n_ki_k}
	\]
	will be abbreviated as \(a_1^{i_1}a_2^{i_2}\cdots a_k^{i_k}\).

	\section{Finite-degree Z-closures and Z*-closures, and our polynomial covering problems}\label{sec:Z-star-clo}
	
	In this section, we will obtain our combinatorial characterization of finite-degree Z*-closures of weight-determined sets, and then proceed to solve Problem~\ref{prob:main}.  We begin with some simple results.  Let \(G\) be a uniform grid.  The following fact is folklore and follows, for instance, from the Footprint bound (see Cox, Little and O'Shea~\cite[Chapter 5, Section 3, Proposition 4]{cox2015ideals}).
	\begin{fact}[Folklore,~\cite{cox2015ideals}]\label{fac:ham-clo}
		Let \(G\) be a uniform grid.  Then \(\zscl_{G,d}([0,r])=[0,N]\), for any \(r,d\in[0,N],\,r\ge d\).
	\end{fact}
	The following result is elementary.
	\begin{proposition}\label{prop:simple-clo}
		If \(E\subseteq[0,N]\) with \(|E|\le d\), then \(\zscl_{G,d}(E)=E\).
	\end{proposition}
	\begin{proof}
		Clearly \(E\subseteq\zscl_{G,d}(E)\).  Also the polynomial \(P(\mb{X})\coloneqq\prod_{t\in E}{\big(\sum_{i\in[0,n]}X_i-t\big)}\) satisfies \(\deg P=|E|\le d\), \(P|_{\ul{E}}=0\) and \(P|_{\ul{j}}\ne0\), for all \(j\not\in E\).  Thus \(\zscl_{G,d}(E)\subseteq E\).
	\end{proof}

	\subsection{Two main lemmas}\label{subsec:main-lem}
	
	We require two main lemmas to obtain our characterization; let us prove them here.
	
	Our first lemma holds over any uniform grid and identifies a collection of `layers' which are certain to lie in the finite-degree Z*-closures of weight-determined sets.
	\begin{lemma}[Closure Builder Lemma]\label{lem:poly-clo-builder}
		Let \(d\in[0,N]\) and \(E\subseteq[0,N]\) with \(|E|\ge d+1\).  Then
		\[
		[0,\min E]\cup[\max E,N]\subseteq\zscl_{G,d}(E).
		\]
	\end{lemma}
	\begin{proof}
		Let \(m=\min E,\,M=\max E\).  Consider any \(j\in[0,m-1]\).  Suppose \(j\not\in\zscl_{G,d}(E)\).  This means \(\ul{j}\not\subseteq\zcl_{G,d}(\ul{E})\).  So there exists \(a\in\ul{j}\) and \(P(\mb{X})\in\mb{R}[\mb{X}]\) such that \(\deg P\le d\), \(P|_{\ul{E}}=0\) and \(P(a)=1\).  Define
		\[
		Q(\mb{X})=P(\mb{X})\cdot\bigg(\prod_{i\in[n]}\prod_{s\in[0,a_i-1]}(X_i-s)\bigg)\bigg(\prod_{t\in[j+1,N]\setminus E}\bigg(\sum_{i\in[0,n]}X_i-t\bigg)\bigg).
		\]
		Then we have \(\deg Q\le d+j+(N-j-|E|)\le N-1\).  Also, \(Q|_{\ul{t}}=0\), for all \(t\in[0,N],\,t\ne j\).  Further, \(Q(x)=0\) if \(x\in\ul{j},\,x\ne a\), and \(Q(a)=a_1!\cdots a_n!\cdot\prod_{t\in[j+1,N]\setminus E}(j-t)\ne0\).  So by Theorem~\ref{thm:alon-furedi}, \(\deg Q\ge N\), a contradiction.  Thus \(j\in\zscl_{G,d}(E)\).  Hence we conclude that \([0,m]\subseteq\zscl_{G,d}(E)\).  Similarly, we can conclude that \([M,N]\subseteq\zscl_{G,d}(E)\).
	\end{proof}
	The following observation will be important for further discussion, and is easy to check.  Recall the set operators \(L_{n,d},\,d\in[0,n]\) from Subsection~\ref{subsec:comb-charac-Z*}.
	\begin{observation}\label{obs:Lnd}
		Let \(G\) be a uniform grid.  For \(d\in[0,N]\) and \(E=\{t_1<\cdots<t_s\}\subseteq[0,N]\), if \(|E|\ge d+1\), then
		\[
		L_{N,d}(E)=[0,t_{s-d}]\cup E\cup[t_{d+1},N]=\bigcup_{A\in\binom{E}{d+1}}\big([0,\min A]\cup A\cup[\max A,n]\big).
		\]
		Consequently, \(\ol{L}_{N,d}(E)\subseteq\zscl_{G,d}(E)\).
	\end{observation}
	\noindent Thus \(L_{n,d}(E)\) represents repeated application of the Closure Builder Lemma~\ref{lem:poly-clo-builder} to \((d+1)\)-subsets of \(E\).
	
	Our second lemma characterizes the finite-degree Z*-closure of \(T_{N,i},\,i\in[0,N]\) in an \SUt\,grid.
	\begin{lemma}\label{lem:T-clo-poly}
		Let \(G\) be an \SUt\,grid.  For every \(i\in[0,N]\),
		\[
		\zscl_{G,d}(T_{N,i})=\begin{cases}
			T_{N,i}&\tx{if }i\le d,\\
			[0,N]&\tx{if }i>d.
		\end{cases}
		\]
	\end{lemma}
	\begin{proof}
		Note that
		\[
		T_{N,\lfloor N/2\rfloor}=\begin{cases}
			[0,N]\setminus\{N/2\}&\tx{if }N\tx{ is even},\\
			[0,N]\setminus\{\lfloor N/2\rfloor,\lfloor N/2\rfloor+1\}&\tx{if }N\tx{ is odd}.
		\end{cases}
		\]
		So if \(i\ge\lfloor N/2\rfloor+1\), then obviously \(T_{N,i}=[0,N]\), and so we are trivially done.  So assume \(i\le\lfloor N/2\rfloor\).  If \(i>d\), then \([0,d]\subseteq T_{N,i}\), and so by Closure Builder Lemma~\ref{lem:poly-clo-builder}, we conclude that \(\zcl_{G,d}(T_{N,i})=[0,N]\).
		
		Now also assume \(i\le d\).  Clearly \(T_{N,i}\subseteq\zscl_{G,d}(T_{N,i})\).  Let \(T'_{N,i}=[0,i-1]\cup\{N-i+1\}\).  Again by Closure Builder Lemma~\ref{lem:poly-clo-builder} and Proposition~\ref{pro:Z*-prop} (b), we get \(\zscl_{G,i}(T_{N,i})=\zscl_{G,i}(T'_{N,i})\).  Also by Proposition~\ref{pro:Z*-prop} (c), we have \(\zscl_{G,d}(T_{N,i})\subseteq\zscl_{G,i}(T_{N,i})\).  So it is enough to prove that \(\zscl_{G,i}(T_{N,i})=T_{N,i}\).
		
		Suppose \(i\in\zscl_{G,i}(T_{N,i})\).  This means \([0,i]\subseteq\zscl_{G,i}(T_{N,i})\).  But then by Lemma~\ref{lem:poly-clo-builder}, and Proposition~\ref{pro:Z*-prop} (b) and (c), we have \([i,N]\subseteq\zscl_{G,i}([0,i])\subseteq\zscl_{G,i}(\zscl_{G,i}(T_{N,i}))\).  This means \(\zscl_{G,i}(T_{N,i})=[0,N]\), that is, \(\zcl_{G,i}(\ul{T_{N,i}})=G\).  Thus we get \(\zcl_{G,i}(\ul{T'_{N,i}})=G\).  Now consider the linear system
		\[
		\sum_{\gamma\in\ul{[0,i]}}c_\gamma\mb{X}^\gamma(a)=0,\quad a\in\ul{T'_{N,i}}.
		\]
		Note that in this system, the variables are \(c_\gamma,\,\gamma\in\ul{[0,i]}\), and the constraints are indexed by \(a\in\ul{T'_{N,i}}\).  So the number of variables is \(|\ul{[0,i]}|=\sqbinom{G}{[0,i]}=\sqbinom{G}{[0,i-1]}+\sqbinom{G}{i}\), and the number of constraints is \(|\ul{T'_{N,i}}|=|\ul{[0,i-1]}|+|\ul{N-i+1}|=\sqbinom{G}{[0,i-1]}+\sqbinom{G}{N-i+1}\).  Since \(G\) is \SUt, we have \(\sqbinom{G}{i}>\sqbinom{G}{i-1}=\sqbinom{G}{N-i+1}\), and therefore this system has a nontrivial solution.  Thus, there exists a nonzero polynomial \(P(\mb{X})\in\spn\{\mb{X}^\gamma:\gamma\in\ul{[0,i]}\}\) such that \(P|_{\ul{T_{N,i}'}}=0\).  By the Combinatorial Nullstellensatz (Theorem~\ref{thm:CN}), we then get \(P\ne0\) in \(V(G)\), that is, \(P(a)\ne0\) for some \(a\in G\).  (Also, obviously by the spanning set description, we get \(\deg P\le i\).)  This means \(a\not\in\zcl_{G,i}(\ul{T'_{N,i}})\), which implies \(\wgt(a)\not\in\zscl_{G,i}(T'_{N,i})\).  So \(\zscl_{G,i}(T'_{N,i})\ne G\), a contradiction.
		
		So we get \(i\not\in\zscl_{G,i}(T_{N,i})\).  Similarly, we can conclude that \(N-i\not\in\zscl_{G,i}(T_{N,i})\).  Now consider any \(j\in[i+1,N-i-1]\).  If \(j\in\zscl_{G,i}(T_{N,i})\), then by the Closure Builder Lemma~\ref{lem:poly-clo-builder} and Observation~\ref{obs:Lnd}, we get \([0,j]\subseteq\zscl_{G,i}(T_{N,i})\), which implies \(i\in\zscl_{G,i}(T_{N,i})\), a contradiction.  Thus we have proven that \([i,N-i]\cap\zscl_{G,i}(T_{N,i})=\emptyset\), that is, \(\zscl_{G,i}(T_{N,i})=T_{N,i}\).  This completes the proof.
	\end{proof}

	\subsection{The main theorem}\label{subsec:combi-charac}
	
	We will now prove Theorem~\ref{thm:main-poly-closure-op}, our combinatorial characterization of finite-degreee Z*-closures of all weight-determined sets in an \SUt\,grid.  We restate Theorem~\ref{thm:main-poly-closure-op} for convenience.
	\begin{reptheorem}[\ref{thm:main-poly-closure-op}]
		Let \(G\) be an \SUt\,grid.  For every \(d\in[0,N]\) and \(E\subseteq[0,N]\),
		\[
		\zscl_{G,d}(E)=\ol{L}_{N,d}(E).
		\]
	\end{reptheorem}
	\begin{proof}
		Fix a \(d\in[0,N]\) and consider any \(E\subseteq[0,N]\).  If \(|E|\le d\), then by Proposition~\ref{prop:simple-clo} and the definition of \(\ol{L}_{N,d}\), we have \(\ol{L}_{N,d}(E)=\zscl_{G,d}(E)=E\).
		
		So now assume that \(|E|\ge d+1\).  By the Closure Builder Lemma~\ref{lem:poly-clo-builder} and Observation~\ref{obs:Lnd}, we already have \(\ol{L}_{N,d}(E)\subseteq\zscl_{G,d}(E)\).  Now consider any \(j\not\in\ol{L}_{N,d}(E)\).  In particular, \(\ol{L}_{N,d}(E)\ne[0,N]\).  We will prove that \(j\not\in\zscl_{G,d}(E)\).  Let
		\[
		i_0=\max\{i\in[0,\lfloor N/2\rfloor]:T_{N,i}\subseteq\ol{L}_{N,d}(E)\}.
		\]
		Note that
		\[
		T_{N,\lfloor N/2\rfloor}=\begin{cases}
			[0,N]\setminus\{N/2\}&\tx{if }N\tx{ is even},\\
			[0,N]\setminus\{\lfloor N/2\rfloor,\lfloor N/2\rfloor+1\}&\tx{if }N\tx{ is odd}.
		\end{cases}
		\]
		So clearly \(T_{N,i_0}=[0,i_0-1]\cup[N-i_0+1,N]\ne[0,N]\).  If \(i_0\ge d+1\), then there exists \(k\in\mb{Z}^+\) such that \([0,d]\subseteq[0,i_0-1]\subseteq L_{N,d}^k(E)\).  Then by Observation~\ref{obs:Lnd}, we have \(L_{N,d}^{k+1}(E)=[0,N]\), that is, \(\ol{L}_{N,d}(E)=[0,N]\), a contradiction.  So we have \(i_0\le d\).
		
		By definition of \(i_0\), it is clear that either \(i_0\not\in\ol{L}_{N,d}(E)\) or \(N-i_0\not\in\ol{L}_{N,d}(E)\).  Without loss of generality, suppose \(i_0\not\in\ol{L}_{N,d}(E)\).  Now we have \(j\in[i_0,N-i_0]\).  By Lemma~\ref{lem:T-clo-poly}, \(\zscl_{G,i_0}(T_{N,i_0})=T_{N,i_0}\), and so \(j\not\in\zscl_{G,i_0}(T_{N,i_0})\).  This means \(\ul{j}\not\subseteq\zcl_{G,i_0}(T_{N,i_0})\).  So there exists \(x_0\in\ul{j}\) and a nonzero \(P(\mb{X})\in\mb{R}[\mb{X}]\) satisfying \(\deg P\le i_0,\,P|_{\ul{T_{N,i_0}}}=0\) and \(P(x_0)=1\).
		
		Let us now show that \(|E\cap[i_0,N-i_0]|\le d-i_0\).  On the contrary, suppose \(|E\cap[i_0,N-i_0]|\ge d-i_0+1\).  Then
		\[
		\big|[N-i_0+1,N]\cup\big(E\cap[i_0,N-i_0]\big)\big|\ge d+1.
		\]
		Let \(E^+=[N-i_0+1,N]\cup\big(E\cap[i_0,N-i_0]\big)\).  Note that since \(T_{N,i_0}\subseteq\ol{L}_{N,d}(E)\), we have \(E^+\subseteq\ol{L}_{N,d}(E)\).  Since \(|E^+|\ge d+1\), by definition of \(\ol{L}_{N,d}\), we have \([0,\min E^+]\subseteq\ol{L}_{N,d}(E)\).  But we have \(i_0\in[0,\min E^+]\).  So in particular, we have \(i_0\in\ol{L}_{N,d}(E)\), which is a contradiction since we assumed \(i_0\not\in\ol{L}_{N,d}(E)\).  Hence \(|E\cap[i_0,N-i_0]|\le d-i_0\).
		
		Now define
		\[
		Q(\mb{X})=\prod_{t\in E\cap[i_0,N-i_0]}\bigg(\sum_{i\in[0,n]}X_i-t\bigg).
		\]
		Then \(\deg Q\le d-i_0\) and \(Q(x_0)\ne0\).  This gives \(\deg PQ\le d,\,PQ|_{\ul{E}}=0\) and \(PQ(x_0)\ne0\), that is, \(PQ|_{\ul{j}}\ne0\).  So \(\ul{j}\not\subseteq\zcl_{G,d}(E)\), which implies \(j\not\in\zscl_{G,d}(E)\).  This completes the proof.
	\end{proof}

	\subsection{Computing the finite-degree Z*-closures efficiently}\label{subsec:clo-eff}
	
	The characterization in Theorem~\ref{thm:main-poly-closure-op} gives a simple algorithm to compute \(\zscl_{G,d}(E)\), for any \(d\in[0,N],\,E\subseteq[0,N]\), in time \(O(N)\) (linear time).  We only consider the complexity of computing the finite-degree Z*-closures modulo the bit complexity of representing the weight-determined sets in the uniform grid \(G\); accomodating the bit complexity will only multiply our bound by a \(\poly(\log N)\) factor.  We will now describe the algorithm.
	
	We will need to consider \emph{shifted} versions of the set operators \(L_{N,d}\) and \(\ol{L}_{N,d}\), defined for subsets of intervals other than \([0,N]\).  Fix any \(a,b\in\mb{Z},\,a\le b\).  For every \(d\in[0,b-a]\), define \(L_{[a,b],d}:2^{[a,b]}\to2^{[a,b]}\) and \(\ol{L}_{[a,b],d}:2^{[a,b]}\to2^{[a,b]}\) as
	\[
	L_{[a,b],d}(E)=L_{b-a,d}(E-a)+a\quad\tx{and}\quad\ol{L}_{[a,b],d}(E)=\ol{L}_{b-a,d}(E-a)+a,\quad\tx{for all }E\subseteq[a,b].
	\]
	Immediately, we have the following analogue of Observation~\ref{obs:Lnd}.
	\begin{observation}\label{obs:Lnd-general}
		Let \(a,b\in\mb{Z},\,a\le b\).  For any \(d\in[0,b-a]\) and \(E=\{t_1<\cdots<t_s\}\subseteq[a,b]\), if \(|E|\ge d+1\), then
		\[
		L_{[a,b],d}(E)=[a,t_{s-d}]\cup E\cup[t_{d+1},b]=\bigcup_{A\in\binom{E}{d+1}}\big([a,\min A]\cup A\cup[\max A,b]\big).
		\]
	\end{observation}
	The following properties of the above set operators will enable us to compute them in linear time.
	\begin{proposition}\label{prop:Lnd-general-prop}
		Let \(a,b\in\mb{Z},\,a\le b\).
		\begin{enumerate}[(a)]
			\item  \(\ol{L}_{[a,b],0}(\emptyset)=\emptyset\), and \(\ol{L}_{[a,b],0}(E)=[a,b]\), for all \(\emptyset\ne E\subseteq[a,b]\).
			\item  \(\ol{L}_{[a,b],b-a}(E)=E\), for all \(E\subseteq[a,b]\).
			\item  If \(b-a\ge 2\), then for any \(d\in[b-a]\),
			\[
			\ol{L}_{[a,b],d}(E)=\begin{cases}
				E&\tx{if }|E|\le d,\\
				\{a,b\}\sqcup\ol{L}_{[a+1,b-1],d-1}(E\setminus\{a,b\})&\tx{if }|E|\ge d+1.
			\end{cases}
			\]
		\end{enumerate}
	\end{proposition}
	\begin{proof}
		Items (a) and (b) are evident from the definitions.  For Item (c), the claim is clear if \(|E|\le d\), again by the definitions.
		
		So now suppose \(|E|\ge d+1\).  It is then immediate that \(\{a,b\}\subseteq\ol{L}_{[a,b],d}(E)\) and so \(\ol{L}_{[a,b],d}(E)=\ol{L}_{[a,b],d}(\{a,b\}\cup E)\).  So let \(E'=\{a,b\}\cup E\).  We only need to show that \(\ol{L}_{[a,b],d}(E')=\{a,b\}\sqcup\ol{L}_{[a+1,b-1],d-1}(E'\setminus\{a,b\})\), since \(\ol{L}_{[a,b],d}(E')=\ol{L}_{[a,b],d}(E)\) and \(E'\setminus\{a,b\}=E\setminus\{a,b\}\).
		
		If \(|E'|=d+1\), then clearly, \(\ol{L}_{[a,b],d}(E')=E'=\{a,b\}\sqcup\ol{L}_{[a+1,b-1],d-1}(E'\setminus\{a,b\})\).
		
		Now assume \(|E'|\ge d+2\).  It is enough to show that \(L_{[a,b],d}^k(E')=\{a,b\}\sqcup L_{[a+1,b-1],d-1}^k(E'\setminus\{a,b\})\), for all \(k\in\mb{N}\).  We will show this by induction on \(k\).  The case \(k=0\) is obvious, since \(L_{[a,b],d}^0\) is the identity operator.  Now suppose the claim is true for some \(k\in\mb{N}\).  Let \(L_{[a,b],d}^k(E')=\{t_1<\ldots<t_s\}\).  So we get
		\begin{align}
			L_{[a,b],d}^{k+1}(E')&=[a,t_{s-d}]\cup L_{[a,b],d}^k(E')\cup[t_{d+1},b]=[a,t_{s-d}]\cup L_{[a+1,b-1],d-1}^k(E'\setminus\{a,b\})\cup[t_{d+1},b],\label{eq:first}
		\end{align}
		where the first equality is by the definition of \(L_{[a,b],d}\), and the second equality is by the induction hypothesis.  But we already have \(t_1=a,\,t_s=b\), and \(L_{[a+1,b-1],d-1}^k(E'\setminus\{a,b\})=\{t_2<\cdots<t_{s-1}\}\).  So
		\begin{align}
			L_{[a+1,b-1],d-1}^{k+1}(E'\setminus\{a,b\})=[a+1,t_{s-d}]\cup L_{[a+1,b-1],d-1}^k(E'\setminus\{a,b\})\cup[t_{d+1},b-1].\label{eq:second}
		\end{align}
		So from (\ref{eq:first}) and (\ref{eq:second}), we get
		\[
		L_{[a,b],d}^{k+1}(E')=\{a,b\}\sqcup L_{[a+1,b-1],d-1}^{k+1}(E'\setminus\{a,b\}).
		\]
		This completes the proof.
	\end{proof}
	
	We then get a straightforward linear time recursive algorithm to compute \(\ol{L}_{[a,b],d},\,d\in[0,b-a]\).  The base case of the recursion appeals to Proposition~\ref{prop:Lnd-general-prop} (a) and (b), and the recursive step appeals to Proposition~\ref{prop:Lnd-general-prop} (c).  The linear run-time is obvious since it is easy to see that there exists a constant \(C>0\) such that for any \(a,b\in\mb{Z},\,a\le b\) and \(d\in[0,b-a],\,E\subseteq[a,b]\), \n(i) \(|E|\) can be computed in time at most \(C(b-a)\), and further \n(ii)  appealing to Observation~\ref{obs:Lnd-general}, \(L_{[a,b],d}(E)\) can be computed in time at most \(C(b-a)\).  So for inputs \(d\in[0,N]\) and \(E\subseteq[0,N]\), the algorithm computes \(\ol{L}_{N,d}(E)\) in time \(O(N)\).
	
	We conclude by stating the pseudocode for this algorithm.  Here we assume that the input \(E\subseteq[a,b]\) is given by its indicator vector \((e_a,\ldots,e_b)\in\{0,1\}^{[a,b]}=\{0,1\}^{b-a+1}\), and \(w\coloneqq\sum_{t\in[a,b]}e_t=|E|\) is already computed in time at most \(C(b-a)\).  Note that Proposition~\ref{prop:Lnd-general-prop} is also the proof of correctness of this algorithm.
	
	\begin{algorithm}[H]\label{algo:main}
		\DontPrintSemicolon
		
		\KwInput{\(a,b\in\mb{N},\,a\le b;\,d\in[0,b-a];\,(e_a,\ldots,e_b)\in\{0,1\}^{[a,b]};\,w\coloneqq\sum_{t\in[a,b]}e_t\)}
		\KwOutput{\((f_a,\ldots,f_b)\in\{0,1\}^{[a,b]}\) such that if \(E\coloneqq\{j\in[a,b]:e_j=1\}\), then \(F\coloneqq\{j\in[a,b]:f_j=1\}\) satisfies \(F=\ol{L}_{[a,b],d}(E)\).}
		\If{\(b=a\)}
		{
			\tbf{return} \((e_a)\)
		}
		\ElseIf{\(b=a+1\)}
		{
			\If{\(d=0\)}
			{
				\If{\((e_a,e_{a+1})=(0,0)\)}
				{
					\tbf{return} \((0,0)\)
				}
				\Else
				{
					\tbf{return} \((1,1)\)
				}
			}
			\Else
			{
				\tbf{return} \((e_a,e_{a+1})\)
			}
		}
		\ElseIf{\(d=b-a\) {\normalfont\tbf{or}} \(w\le d\)}
		{
			\tbf{return} \((e_a,\ldots,e_b)\)
		}
		\Else
		{
			\tbf{return} \(\big(1,\texttt{L-bar}\big(a+1,b-1,d-1,(e_{a+1},\ldots,e_{b-1}),w-(e_a+e_b)\big),1\big)\)
		}
		\caption{\texttt{L-bar}: Computing \(\ol{L}_{[a,b],d}\)}
	\end{algorithm}

	\subsection{Solving our polynomial covering problems}\label{aubsec:solve-poly}
	
	Let us gather the work done so far to solve our polynomial covering problems (Problem~\ref{prob:main} (b)).  Let us begin by proving Lemma~\ref{lem:poly-cover-defn} that relates our polynomial covering problems with the finite-degree Z*-closures.
	\begin{proof}[Proof of Lemma~\ref{lem:poly-cover-defn}]
		Let \(G\) be a uniform grid and consider \(E\subsetneq[0,N]\).
		\begin{enumerate}[(a)]
			\item  We have the following equivalences.
			\begin{align*}
				&d'=\PC_G(E)\\
				\iff\quad&\tx{There exists }P(\mb{X})\in\mc{I}(\ul{E})_{d'}\tx{ such that }\mc{Z}_G(P)\ne G\\
				&\quad\tx{and for every }Q(\mb{X})\in\mc{I}(\ul{E})_{d'-1},\tx{ we have }\mc{Z}_G(P)=G\\
				\iff\quad&\zcl_{G,d'}(\ul{E})\ne G\tx{ and }\zcl_{G,d'-1}(\ul{E})=G\\
				\iff\quad&\zscl_{G,d'}(E)\ne[0,N]\tx{ and }\zscl_{G,d'-1}(E)=[0,N]\\
				\iff\quad&d'=\min\{d\in[0,N]:\zscl_{G,d'}(E)\ne[0,N]\}.
			\end{align*}
			
			\item  Let \(d'=\PPC_G(E)\) and \(d''=\min\{d\in[0,N]:\zscl_{G,d}(E)=E\}\).  We have the following implications.
			\begin{align*}
				&d'=\PPC_G(E)\\
				\implies\quad&\tx{There exists }P(\mb{X})\in\mc{I}(\ul{E})_{d'}\tx{ such that }\ul{j}\not\subseteq\mc{Z}_G(P),\tx{ for every }j\in[0,N]\setminus E\\
				\implies\quad&\ul{j}\not\subseteq\zcl_{G,d'}(\ul{E}),\tx{ for every }j\in[0,N]\setminus E\\
				\implies\quad&j\not\in\zscl_{G,d'}(E),\tx{ for every }j\in[0,N]\setminus E,\tx{ that is, }\zscl_{G,d'}(E)=E\\
				\implies\quad&d''\le d'.
			\end{align*}
			Now since \(d''=\min\{d\in[0,N]:\zscl_{G,d}(E)=E\}\), we have \(\ul{j}\not\subseteq\zcl_{G,d''}(\ul{E})\), for every \(j\in[0,N]\setminus E\).  For every \(j\in[0,N]\setminus E\), choose \(a^{(j)}\in\ul{j}\setminus\zcl_{G,d''}(\ul{E})\); so there exists \(P_j(\mb{X})\in\mb{R}[\mb{X}]\) such that \(\deg P_j\le d''\), \(P_j|_{\ul{E}}=0\) and \(P_j(a^{(j)})=1\).  Then we can choose \(\beta_j\in\mb{R},\,j\in[0,N]\setminus E\) such that the polynomial \(P(\mb{X})\coloneqq\sum_{j\in[0,N]\setminus E}\beta_jP_j(\mb{X})\) satisfies \(\deg P\le d''\), \(P|_{\ul{E}}=0\) and \(P(a^{(j)})\ne0\), for all \(j\in[0,N]\setminus E\).  This implies \(d'\le d''\), and therefore completes the proof.\qedhere
		\end{enumerate}
	\end{proof}
	
	Let us now proceed to prove Proposition~\ref{prop:poly-L-finer}.  We require the following lemma.
	\begin{lemma}\label{lem:admitting-lemma}
		Let \(d\in[0,N],\,i\in[0,d]\).  Then \(E\subseteq[0,N]\) is \((d,i)\)-admitting if and only if \(L_{N,d}(E)\) is \((d,i)\)-admitting.
	\end{lemma}
	\begin{proof}
		Fix any \(d\in[0,N],\,i\in[0,d]\) and \(E\subseteq[0,N]\).  We clearly have the following implications.
		\begin{align*}
			&L_{N,d}(E)\tx{ is \((d,i)\)-admitting}\\
			\implies\quad&L_{N,d}(E)\cup T_{N,i}\ne[0,N]\tx{ and }|L_{N,d}(E)\setminus T_{N,i}|\le d-i\\
			\implies\quad&E\cup T_{N,i}\subseteq L_{N,d}(E)\cup T_{N,i}\ne[0,N]\tx{ and }|E\setminus T_{N,i}|\le|L_{N,d}(E)\setminus T_{N,i}|\le d-i\\
			\implies\quad&E\tx{ is \((d,i)\)-admitting}.
		\end{align*}
		Conversely suppose \(E\) is \((d,i)\)-admitting.  If \(|E|\le d\), then \(L_{N,d}(E)=E\) and we are done.  So now assume \(|E|\ge d+1\).  We have \(E\cup T_{N,i}\ne[0,N]\) and \(|E\setminus T_{N,i}|\le d-i\).  This implies
		\begin{align}
			&|E\cap[0,N-i]|\le|[0,i-1]\cup(E\setminus T_{N,i})|\le d,\label{ineq:left}\\
			\tx{and}\quad&|E\cap[i,N]|\le|[N-i+1,N]\cup(E\setminus T_{N,i})|\le d.\label{ineq:right}
		\end{align}
		Enumerate \(E=\{t_1<\cdots<t_s\}\).  Then \(L_{N,d}(E)=[0,t_{s-d}]\cup E\cup[t_{d+1,N}]\), since \(|E|\ge d+1\).  Further, Inequality~(\ref{ineq:left}) implies \(t_{d+1}\ge N-i+1\), and Inequality~(\ref{ineq:right}) implies \(t_{s-d}\le i-1\).  So we get
		\begin{enumerate}[(a)]
			\item  \(L_{N,d}(E)\subseteq [0,i-1]\cup E\cup[N-i+1]=E\cup T_{N,i}\ne[0,N]\), and
			\item  \(L_{N,d}(E)\cap[i,N-i]=E\cap[i,N-i]\), which gives \(|L_{N,d}(E)\setminus T_{N,i}|=|E\setminus T_{N,i}|\le d-i\).
		\end{enumerate}
		Thus \(L_{N,d}(E)\) is \((d,i)\)-admitting.
	\end{proof}
	
	We are now ready to prove Proposition~\ref{prop:poly-L-finer}.
	\begin{proof}[Proof of Proposition~\ref{prop:poly-L-finer}]
		\begin{enumerate}[(a)]
			\item  We need to prove that \(\ol{L}_{N,d}(E)\ne[0,N]\) if and only if \(E\) is \(d\)-admitting.
			
			Suppose \(\ol{L}_{N,d}(E)\ne[0,N]\).  Let
			\[
			i_0=\max\{i\in[0,\lfloor N/2\rfloor]:T_{N,i}\subseteq\ol{L}_{N,d}(E)\}.
			\]
			If \(i_0\ge d+1\), then there exists \(k\in\mb{Z}^+\) such that \([0,d]\subseteq[0,i_0-1]\subseteq L_{N,d}^k(E)\).  Then by definition of \(L_{N,d}\) and Lemma~\ref{lem:poly-clo-builder}, we have \(L_{N,d}^{k+1}(E)=[0,N]\), thus implying \(\ol{L}_{N,d}(E)=[0,N]\), a contradiction.  So we have \(i_0\le d\).  Also, without loss of generality, let \(i_0\not\in\ol{L}_{N,d}(E)\).  Clearly \(E\cup T_{N,i_0}\subseteq\ol{L}_{N,d}(E)\ne[0,N]\).
			
			Now suppose \(|E\setminus T_{N,i_0}|\ge d-i_0+1\), then \(|E\cap[i_0,N]|\ge d+1\).  Let \(m=\min(E\cap[i_0,N])\ge i_0\).  By Observation~\ref{obs:Lnd}, we get \([0,m]\subseteq \ol{L}_{N,d}(E)\); in particular, we get \(i_0\in\ol{L}_{N,d}(E)\), which is a contradiction.  So \(|E\setminus T_{N,i_0}|\le d-i_0\).  Thus \(E\) is \((d,i_0)\)-admitting, and hence \(E\) is \(d\)-admitting.
			
			Conversely, suppose \(E\) is \(d\)-admitting.  Let \(k_0\in\mb{Z}^+\) be the least such that \(\ol{L}_{N,d}(E)=L_{N,d}^{k_0}(E)\).  So applying Lemma~\ref{lem:admitting-lemma} precisely \(k_0\) times, we conclude that \(\ol{L}_{N,d}(E)\) is \(d\)-admitting.  So there exists \(i\in[0,d]\) such that \(\ol{L}_{N,d}(E)\cup T_{N,i}\ne[0,N]\), which implies \(\ol{L}_{N,d}(E)\ne[0,N]\).
			
			\item  We need to prove that \(\ol{L}_{N,d}(E)=E\) if and only if \(T_{N,|E|-d}\subseteq E\).  Firstly, note that by definition of \(\ol{L}_{N,d}\), we have \(\ol{L}_{N,d}(E)=E\) if and only if \(L_{N,d}(E)=E\).  Secondly, note that the assertion is vacuously true if \(|E|\le d\).  So now assume \(|E|\ge d+1\).  Let \(E=\{t_1<\cdots<t_s\}\).
			
			Now suppose \(L_{N,d}(E)=E\).  Let \(E'=\{t_1,\ldots,t_{d+1}\}\).  By definition of \(L_{N,d}\), we have \([0,t_1]\cup E'\cup[t_{d+1},N]\subseteq L_{N,d}(E')\subseteq\ol{L}_{N,d}(E)=E\).  So \(\{t_{d+1},\ldots,t_s\}=[t_{d+1},N]\), which implies \(N-t_{d+1}+1=s-d\), that is, \(t_{d+1}=N-s+d+1=N-(|E|-d)+1\).  Thus \([N-(|E|-d)+1,N]\subseteq E\).  Further, let \(E''=\{t_{s-d},\ldots,t_s\}\).  Again, by definition of \(L_{N,d}\), we have \([0,t_{s-d}]\cup E''\cup[t_s,N]\subseteq L_{N,d}(E'')\subseteq\ol{L}_{N,d}(E)=E\).  So \(\{t_1,\ldots,t_{s-d}\}=[0,t_{s-d}]\), which implies \(t_{s-d}+1=s-d\), that is, \(t_{s-d}=s-d-1=(|E|-d)-1\).  Thus \([0,(|E|-d)-1]\subseteq E\).  Hence \(T_{N,|E|-d}\subseteq E\).
			
			Conversely, suppose \(T_{N,|E|-d}\subseteq E\).  So \([0,(|E|-d)-1]=\{t_1,\ldots,t_{s-d}\}\) and \([N-(|E|-d)+1,N]=\{t_{d+1},\ldots,t_s\}\), that is, \(t_{s-d}=(|E|-d)-1\) and \(t_{d+1}=N-(|E|-d)+1\).  Hence \(L_{N,d}(E)=[0,t_{s-d}]\cup E\cup[t_{d+1},N]=E\).\qedhere
		\end{enumerate}
	\end{proof}
	We have thus proved Theorem~\ref{thm:main-poly-cover}, which is our solution to Problem~\ref{prob:main} (b) for \SUt\,grids.

	\section{Finite-degree h-closures and our hyperplane covering problems}\label{sec:h-clo-hyper}
	
	We introduce another new closure operator, defined using polynomials representing hyperplane covers.  Let
	\[
	\ms{H}_n=\bigg\{\prod_{i=1}^k\ell_i(\mb{X}):k\in\mb{N};\,\ell_i(\mb{X})\in\mb{R}[\mb{X}]\tx{ and }\deg\ell_i\le1,\tx{ for all }i\in[k]\bigg\}.
	\]
	Let \(G\) be a uniform grid.  For any \(d\in[0,N]\) and \(S\subseteq G\), we define the \tsf{degree-\(d\) h-closure} of \(S\) as \(\hcl_{G,d}(S)=\mc{Z}(\mc{I}(\ms{H}_n,S)_d)\).  We will focus on the case of the Boolean cube.  It is immediate that the finite-degree h-closure of a symmetric set is symmetric, and so we will use our indentification of symmetric sets of \(\{0,1\}^n\) with subsets of \([0,n]\).
	
	Our main result in this section is a characterization of finite-degree h-closures of all symmetric sets of the Boolean cube.  In fact, we prove that these coincide with the finite-degree Z-closures (and Z*-closures).  Let us first show that Lemma~\ref{lem:poly-clo-builder} and Lemma~\ref{lem:T-clo-poly} have analogues for finite-degree hyperplane closures.
	\begin{lemma}\label{lem:main-lem-hcl}
		\begin{enumerate}[(a)]
			\item  {\normalfont(Closure Builder Lemma)}\quad  Let \(d\in[0,n]\) and \(E\subseteq[0,n]\) with \(|E|\ge d+1\).  Then
			\[
			[0,\min E]\cup[\max E,n]\subseteq\hcl_{n,d}(E).
			\]
			\item  For every \(i\in[0,N]\),
			\[
			\hcl_{n,d}(T_{n,i})=\begin{cases}
				T_{n,i}&\tx{if }i\le d,\\
				[0,n]&\tx{if }i>d.
			\end{cases}
			\]
		\end{enumerate}
	\end{lemma}
	\begin{proof}
		\begin{enumerate}[(a)]
			\item  The proof is similar to that of Lemma~\ref{lem:poly-clo-builder}; instead of considering polynomials in \(\mb{R}[\mb{X}]\), we just need to consider polynomials in \(\ms{H}_n\) throughout.
			
			\item  We only need to consider \(i\le\min\{d,\lfloor n/2\rfloor\}\).  The other case can be argued exactly as in the proof of Lemma~\ref{lem:T-clo-poly}.  Now consider the polynomial
			\[
			P(\mb{X})=(X_1-X_2)\cdots(X_{2i-1}-X_{2i}).
			\]
			Clearly \(\deg P=i\le d\).  For any \(x\in\ul{j}\), where \(j\in[0,i-1]\), there exists \(t\in[i]\) such that \(x_{2t-1}=x_{2t}=0\); this gives \(P(x)=0\).  For any \(x\in\ul{j}\), where \(j\in[N-i+1,N]\), there exists \(t\in[i]\) such that \(x_{2t-1}=x_{2t}=1\); this gives \(P(x)=0\).  So \(P|_{\ul{T_{n,i}}}=0\).  Now consider any \(j\in[i,n-i]\).  Let \(x^{(j)}=(10)^i1^{j-i}0^{n-j-1}\in\ul{j}\).  Then we have \(P(x^{(j)})=1\).  This implies \(j\not\in\hcl_{n,d}(T_{n,i})\).  Hence \(\hcl_{n,d}(T_{n,i})=T_{n,i}\).\qedhere
		\end{enumerate}
	\end{proof}
	
	We can now prove Theorem~\ref{thm:hcl=zcl}.
	\begin{proof}[Proof of Theorem~\ref{thm:hcl=zcl}]
		The proof is similar to that of Theorem~\ref{thm:main-poly-closure-op}.  Instead of considering polynomials in \(\mb{R}[\mb{X}]\), we need to consider polynomials in \(\ms{H}_n\) throughout.  In addition, we need to replace Lemma~\ref{lem:poly-clo-builder} and Lemma~\ref{lem:T-clo-poly} with Lemma~\ref{lem:main-lem-hcl} (a) and (b) respectively, throughout.
	\end{proof}
	
	By Observation~\ref{obs:hyp-cover-defn} (a), Theorem~\ref{thm:hcl=zcl} and Proposition~\ref{prop:poly-L-finer} (a), we have proved Theorem~\ref{thm:main-hyp-cover} (a).  
	
	\begin{observation}\label{obs:tight-construction}
		From the proof of Lemma~\ref{lem:main-lem-hcl} (b), showing \(\hcl_{n,d}(T_{n,i})=T_{n,i}\) for \(i\le d\), we can infer the stronger statement:  For \(i,d\in[0,n],\,i\le d\), there exists \(P(\mb{X})\in\ms{H}_n\) such that \(P|_{\ul{T_{n,i}}}=0\) and \(P|_{\ul{j}}\ne0\), for every \(j\in[i,n-i]\).
	\end{observation}
	By Observation~\ref{obs:hyp-cover-defn} (b), Theorem~\ref{thm:hcl=zcl}, Proposition~\ref{prop:poly-L-finer} (b) and Observation~\ref{obs:tight-construction}, we have proved Theorem~\ref{thm:main-hyp-cover} (b).  This completes our solution to Problem~\ref{prob:main} (a) in the Boolean cube setting.

	\section{Other applications}\label{sec:further-appl}
	
	Our combinatorial characterization of finite-degree Z*-closures of weight-determined sets in \SUt\,grids (Theorem~\ref{thm:main-poly-closure-op}) may also be interesting in its own right.  Indeed, we will consider two other applications in this section.
	
	\subsection{An easy proof of a lemma by Alon et al.~\cite{alon1988balancing}}\label{subsec:Alon-etal}
	
	Consider the following simple fact; the proof is obvious.
	\begin{fact}\label{fact:port-01-to-pm1}
		Let \(G\) be a uniform grid, \(\theta:\mb{R}^n\to\mb{R}^n\) be an invertible affine linear transformation and \(\mc{G}=\theta(G)\).  Note that \(\theta\) induces an obvious invertible affine linear transformation \(\theta:\mb{R}[\mb{X}]\to\mb{R}[\mb{X}]\) (by abuse of notation). 
		\begin{enumerate}[(a)]
			\item  \(\mc{Z}_\mc{G}(I)=\theta(\mc{Z}_G(\theta^{-1}(I)))\), for every \(I\subseteq\mb{R}[\mb{X}]\).
			\item  Note that for any \(d\in[0,N]\), we have \(\theta(\mb{R}[\mb{X}]_d)=\mb{R}[\mb{X}]_d\).  So by Item (a),
			\[
			\mc{Z}_\mc{G}(\mb{R}[\mb{X}]_d)=\theta(\mc{Z}_G(\mb{R}[\mb{X}]_d)),\quad\tx{for every }d\in[0,N].
			\]
		\end{enumerate}
	\end{fact}
	\noindent Now suppose we represent the Boolean cube as \(\{-1,1\}^n\).  In this case, we define the \tsf{Hamming weight} of \(x\in\{-1,1\}^n\) as \(|x|=|\{i\in[n]:x_i=-1\}|\).  By Fact~\ref{fact:port-01-to-pm1}, we can therefore \emph{port} all our results to the setting of the Boolean cube \(\{-1,1\}^n\) by considering \(\theta:\{0,1\}^n\to\{-1,1\}^n\) defined as \(\theta(x)=1^n-2x,\,x\in\{0,1\}^n\).
	
	Appealing to Fact~\ref{fact:port-01-to-pm1}, Lemma~\ref{lem:alon-lemma} states that \(\zscl_{n,n/2-1}(E_0)=\zscl_{n,n/2-1}(E_1)=[0,n]\).  This is equivalent to Proposition~\ref{prop:ABCO-equivalent}, by Theorem~\ref{thm:main-poly-closure-op}.  In fact, we can prove the following slightly more general statement; Proposition~\ref{prop:ABCO-equivalent} is a special case.
	\begin{proposition}\label{prop:ABCO-general}
		Let \(G\) be a uniform grid and \(m\in[N]\).  Let \(E_{m,i}=\{j\in[0,N]:j\equiv i\modulo{m}\},\,i\in[0,m-1]\).  Then \(\ol{L}_{N,\lfloor N/m\rfloor-1}(E_{m,i})=[0,N]\), for all \(i\in[0,m-1]\).
	\end{proposition}
	\noindent For simplicity, let us just prove Proposition~\ref{prop:ABCO-general} for the case of \(N\) being an even positive integer, and \(m=2\).  The general case can be proven along similar lines.
	\begin{proof}[Proof of Proposition~\ref{prop:ABCO-general} (when \(N\) is even and \(m=2\))]
		Let \(N=2k,\,k\in\mb{Z}^+\).  Since \(m=2\), we need to prove that \(\ol{L}_{2k,k-1}(E_{2,0})=\ol{L}_{2k,k-1}(E_{2,1})=[0,2k]\).  Let us prove \(\ol{L}_{2k,k-1}(E_{2,0})=[0,2k]\); the other claim can be proved in an analogous way.
		
		Our argument is an illustration of Algorithm~\ref{algo:main}, by using Proposition~\ref{prop:Lnd-general-prop}.  Recall the set operators \(L_{[a,b],d}:2^{[a,b]}\to2^{[a,b]}\) and \(\ol{L}_{[a,b],d}:2^{[a,b]}\to2^{[a,b]}\), for \(a,b\in\mb{Z},\,a\le b\) and \(d\in[0,b-a]\), defined in Subsection~\ref{subsec:clo-eff}.  So \(\ol{L}_{2k,k-1}(E_{2,0})=\ol{L}_{[0,2k],k-1}(E_{2,0})\).  For \(i\in[k]\), let \(F_i=E_{2,0}\cap[k-i,k+i]\).  It is easy to see that for each \(i\in[k]\), we have
		\[
		|F_i|=\begin{cases}
			i&\tx{if \(i\) is odd},\\
			i+1&\tx{if \(i\) is even}.
		\end{cases}
		\]
		
		We will prove, by induction, that \(\ol{L}_{[k-i,k+i],i-1}(F_i)=[k-i,k+i]\), for all \(i\in[k]\).  By Proposition~\ref{prop:Lnd-general-prop} (a), we get the base case as \(\ol{L}_{[k-1,k+1],0}(F_1)=[k-1,k+1]\).  Now assume \(\ol{L}_{[k-i,k+i],i-1}(F_i)=[k-i,k+1]\), for some \(i\in[k-1]\).  Note that we have \(F_i=F_{i+1}\setminus\{k-i-1,k+i+1\}\). So by Proposition~\ref{prop:Lnd-general-prop} (c) and the induction hypothesis, we get
		\[
		\ol{L}_{[k-i-1,k+i+1],i}(F_{i+1})=\{k-i-1,k+i+1\}\cup\ol{L}_{[k-i,k+i],i-1}(F_i)=[k-i-1,k+i+1].
		\]
		This completes the proof.
	\end{proof}
	
	\subsection{Certifying degrees of weight-determined sets}\label{subsec:certdeg}
	
	Recall that for a uniform grid \(G\) and subset \(S\subseteq G\), the certifying degree \(\certdeg(S)\) is defined to be the smallest \(d\in[0,N]\) such that \(S\) has a certifying polynomial with degree at most \(d\).  By this definition, we observe that
	\[
	\certdeg(S)=\min\{d\in[0,N]:\zcl_{G,d}(S)\ne G\}.
	\]
	Thus for any weight-determined set \(\ul{E},\,E\subsetneq[0,N]\), we get
	\[
	\certdeg(\ul{E})=\min\{d\in[0,N]:\zcl_{G,d}(\ul{E})\ne G\}=\min\{d\in[0,N]:\zscl_{G,d}(E)\ne[0,N]\},
	\]
	since \(\zcl_{G,d}(\ul{E})=G\) if and only if \(\zscl_{G,d}(E)=[0,N]\).
	Consider any symmetric subset \(E\subsetneq[0,n]\).  It then follows immediately from Lemma~\ref{lem:poly-cover-defn}, Theorem~\ref{thm:main-poly-closure-op} and Proposition~\ref{prop:poly-L-finer} that if \(G\) is an \SUt\,grid, then for any \(E\subsetneq[0,N]\),
	\[
	\certdeg(\ul{E})=\PC_G(E)=\min\{d\in[0,N]:E\tx{ is \(d\)-admitting}\}.
	\]
	This proves Theorem~\ref{thm:certdeg}.

	\section{A third variant: the exact covering problem}\label{sec:exact-cover}
	
	Our third covering problem is quite an intuitive variant of the hyperplane and polynomial covering problems, given the \emph{nontrivial} and \emph{proper} covering versions that we have considered so far.  However, we have more questions than answers about this third variant.  Let \(G\) be a uniform grid.  Consider a weight-determined set \(\ul{E}\), where \(E\subsetneq[0,N]\).  We say
	\begin{itemize}
		\item  a family of hyperplanes \(\mc{H}\) in \(\mb{R}^n\) is an \tsf{exact hyperplane cover} of \(\ul{E}\) if \(\ul{E}=\big(\bigcup_{H\in\mc{H}}H\big)\cap G\).
		\item  a polynomial \(P(\mb{X})\in\mb{R}[\mb{X}]\) is  an \tsf{exact polynomial cover} of \(\ul{E}\) if \(\ul{E}=\mc{Z}_G(P)\).
	\end{itemize}
	Let \(\EHC_G(E)\) and \(\EPC_G(E)\) denote the minimum size of an exact hyperplane cover and the minimum degree of an exact polynomial cover respectively, for a weight-determined set \(\ul{E},\,E\subsetneq[0,N]\).  In the case \(G=\{0,1\}^n\), we will instead use the notations \(\EHC_n(E)\) and \(\EPC_n(E)\).
	
	We first note that \(\EPC_G\) can be characterized in terms of the finite-degree Z-closures. Contrast this with Lemma~\ref{lem:poly-cover-defn} which characterizes \(\PPC_G\) in terms of the finite-degree Z*-closures.
	\begin{proposition}\label{prop:EPC-grid}
		Let \(G\) be a uniform grid.  For any \(E\subsetneq[0,N]\),
		\[
		\EPC_G(E)=\min\{d\in[0,N]:\zcl_{G,d}(\ul{E})=\ul{E}\}.
		\]
	\end{proposition}
	\begin{proof}[Proof of Proposition~\ref{prop:EPC-grid}]
		Let \(d'=\EPC_G(E)\) and \(d''=\min\{d\in[0,N]:\zcl_{G,d}(\ul{E})=\ul{E}\}\).  There exists a polynomial \(P(\mb{X})\in\mb{R}[\mb{X}]\) such that \(\deg P=d'\), \(P|_{\ul{E}}=0\) and \(P(a)\ne0\), for all \(a\in G\setminus\ul{E}\).  This implies \(\zcl_{G,d'}(\ul{E})=\ul{E}\), and so \(d''\le d'\).
		
		Further, for every \(a\in G\setminus\ul{E}\), there exists \(Q_a(\mb{X})\in\mb{R}[\mb{X}]\) such that \(\deg Q_a\le d''\), \(Q_a|_{\ul{E}}=0\) and \(Q_a(a)=1\).  We can then choose scalars \(\beta_a\in\mb{R},\,a\in G\setminus\ul{E}\) such that the polynomial \(Q(\mb{X})\coloneqq\sum_{a\in G\setminus\ul{E}}\beta_aP_a(\mb{X})\) satisfies \(\deg P\le d''\), \(P|_{\ul{E}}=0\) and \(P(a)\ne0\), for all \(a\in G\setminus\ul{E}\).  So \(d'\le d''\), and this completes the proof.
	\end{proof}
	
	However, we do not have a further characterization of the finite-degree Z-closures of weight-determined sets.  In the Boolean cube setting, however, since the finite-degree Z-closures and Z*-closures coincide for all symmetric sets, we immediately get the following by appealing to Theorem~\ref{thm:main-poly-cover}, Theorem~\ref{thm:main-hyp-cover}, Theorem~\ref{thm:main-poly-closure-op} and Theorem~\ref{thm:hcl=zcl}.
	\begin{corollary}\label{cor:Boolean-all}
		Consider the Boolean cube \(\{0,1\}^n\).  For any \(E\subsetneq[0,n]\),
		\[
		\EPC_n(E)=\PPC_n(E)=\PHC_n(E)=|E|-\max\{i\in[0,n]:T_{n,i}\subseteq E\}.
		\]
	\end{corollary}
	\noindent Further, characterizing \(\EHC_n(E)\) for \(E\subsetneq[0,n]\) seems to be even more difficult.  We have the following partial results.
	\begin{proposition}\label{prop:EHC}
		Consider the Boolean cube \(\{0,1\}^n\), and any \(E\subsetneq[0,n]\).
		\begin{enumerate}[(a)]
			\item  If \(T_{n,1}\not\subseteq E\), then \(\EHC_n(E)=|E|\).
			\item  If \(n\ge 2\) and \(T_{n,1}\subseteq E,\,T_{n,2}\not\subseteq E\), then \(\EHC_n(E)=|E|-1\).
			\item  If \(n\ge4\), then \(\EHC_n(T_{n,2})=2\).
		\end{enumerate}
	\end{proposition}
	\begin{proof}
		\begin{enumerate}[(a)]
			\item  Clearly \(\EHC_n(E)\ge\PPC_n(E)=|E|-\max\{i\in[0,n]:T_{n,i}\subseteq E\}=|E|\), since \(T_{n,1}\not\subseteq E\).  Further, the hyperplane cover \(\{H_t:t\in E\}\), where \(H_t(\mb{X})\coloneqq\sum_{i\in[n]}X_i-t,\,t\in E\), is an exact cover of \(\ul{E}\) having size \(|E|\).
			
			\item  Again clearly \(\EHC_n(E)\ge\PPC_n(E)=|E|-\max\{i\in[0,n]:T_{n,i}\subseteq E\}=|E|-1\), since \(T_{n,1}\subseteq E,\,T_{n,2}\not\subseteq E\).  Further, we can choose scalars \(a_i\in\mb{R},\,i\in[n]\) such that \(H_0(\mb{X})\coloneqq\sum_{i\in[n]}a_iX_i\) satisfies \(\mc{Z}_G(H_0)=\{0^n,1^n\}\).  So the hyperplane cover \(\{H_0\}\cup\{H_t:t\in E\setminus T_{n,1}\}\), where \(H_t(\mb{X})\coloneqq\sum_{i\in[n]}X_i-t,\,t\in E\setminus T_{n,1}\), is an exact cover of \(\ul{E}\) having size \(|E|-1\).
			
			\item  Obviously \(\EHC_n(T_{n,2})\ge\PPC_n(T_{n,2})=2\).  Now consider the hyperplane cover \(\{H_0,H_1\}\), where
			\[
			H_0(\mb{X})\coloneqq\n-(n-2)X_2+\sum_{i\in[3,n]}X_i\quad\tx{and}\quad H_1(\mb{X})\coloneqq\n-(n-3)X_1+\sum_{i\in[2,n]}X_i-1.
			\]
			It is clear that \(\{0^n,1^n,10^{n-1},01^{n-1}\}\subseteq\mc{Z}(H_0)\), and \(\{0^j10^{n-j-1},1^j01^{n-j-1}\}\subseteq\mc{Z}(H_1)\), for every \(j\in[2,n]\).  Let \(x\in\ul{[2,n-2]}\).  If \(x_2=0\) then \(H_0(x)=\sum_{i\in[3,n]}x_i\ge1\), and if \(x_2=1\) then \(H_0(x)=-(n-2)+\sum_{i\in[3,n]}x_i\le-1\).  Further, if \(x_1=0\) then \(H_1(x)=\sum_{i\in[2,n]}x_i-1\ge1\), and if \(x_1=1\) then \(H_1(x)=-(n-3)+\sum_{i\in[2,n]}x_i-1\le-1\).  Thus \(\{H_0,H_1\}\) is an exact hyperplane cover of \(\ul{T_{n,2}}\), with size 2.  This implies that \(\EHC_n(T_{n,2})=2=|T_{n,2}|-2\).\qedhere
		\end{enumerate}
	\end{proof}
	
	Now assume \(n\ge4\), and let \(E\subsetneq[0,n]\) such that \(T_{n,2}\subseteq E\).  Consider \(\{H_0,H_1\}\), the exact hyperplane cover of \(\ul{T_{n,2}}\), as given in the proof of Proposition~\ref{prop:EHC} (c).  Then \(\{H_0,H_1\}\cup\{H_j:j\in E\setminus T_{n,2}\}\) is an exact hyperplane cover for \(\ul{E}\), with size \(|E|-2\), where \(H_j(\mb{X})\coloneqq\sum_{i\in[n]}X_i-j\), for all \(j\in E\setminus T_{n,2}\).  So \(\EHC_n(E)\le|E|-2\).  We conjecture that this bound is tight, for every \(E\subsetneq[0,n]\) such that \(T_{n,2}\subseteq E\).
	\begin{conjecture}\label{conj:EHC}
		For \(n\ge4\), consider the Boolean cube \(\{0,1\}^n\).  If \(E\subsetneq[0,n]\) such that \(T_{n,2}\subseteq E\), then \(\EHC_n(E)=|E|-2\).
	\end{conjecture}
	Finally, we propose the following open question.
	\begin{openproblem}
		For a uniform (or \SUt) grid \(G\ne\{0,1\}^n\), determine \(\EHC_G(E)\) and \(\EPC_n(E)\), for all \(E\subsetneq[0,N]\).
	\end{openproblem}

	\newpage
	\section{Further remarks: Some updates post publication}\label{sec:updates}
	
	This section was added post the publication of this work.  Here, we mention a few remarks that complete this work better.
	
	\subsection{The exact covering problem: Conjecture~\ref{conj:EHC} is false, and \(\EHC_n=\EPC_n\)}
	
	Consider the Boolean cube \(\{0,1\}^n\).  We will characterize \(\EHC_n(E)\), for all \(E\subsetneq[0,n]\).  By the definitions and Corollary~\ref{cor:Boolean-all}, we get
	\[
	\EHC_n(E)\ge\EPC_n(E)=\PPC_n(E)=\PHC_n(E)=|E|-\max\{i\in[0,n]:T_{n,i}\subseteq E\}.
	\]
	In Proposition~\ref{prop:EHC} (a) and (b), we see that (for all valid \(n\)) the first inequality above is tight in the cases \n(i) \(T_{n,1}\not\subseteq E\), and \n(ii) \(T_{n,1}\subseteq E,\,T_{n,2}\not\subseteq E\).  The only remaining case is \(T_{n,2}\subseteq E\).  In this case, Proposition~\ref{prop:EHC} (c) and Conjecture~\ref{conj:EHC} respectively state that the first inequality is tight if \(E=T_{n,2}\), and not tight otherwise.  Now, we show that the first inequality is tight if \(T_{n,2}\subsetneq E\).  In particular, Conjecture~\ref{conj:EHC} is false.
	
	In a recent work, Ghosh, Kayal, and Nandi~\cite{ghosh-kayal-nandi-2022-alon-furedi} gave the following construction of hyperplanes, while solving a different hyperplane covering problem.  This construction, in fact, extends our construction given in the proof of Proposition~\ref{prop:EHC} (c).  We state this in our notation, which agrees with~\cite{ghosh-kayal-nandi-2022-alon-furedi} up to a change of coordinates.
	\begin{lemma}[{\cite[Lemma 2.5]{ghosh-kayal-nandi-2022-alon-furedi}}]\label{lem:GKN}
		  Let \(i\in[0,\lceil n/2\rceil]\).  Consider the hyperplane cover \(\{H_j:j\in[0,i-1]\}\), where
		  \[
		  H_j(\mb{X})\coloneqq-(n-2i+j)X_{i-1-j}+\sum_{t\in[i-j,n]}X_t-j.
		  \]
		  Then \(\{H_j:j\in[0,i-1]\}\) is an exact hyperplane cover of \(\ul{T_{n,i}}\), having size \(i\).
	\end{lemma}
	We then immediately get the following.
	\begin{corollary}\label{cor:EHC}
		Consider the Boolean cube \(\{0,1\}^n\).  For any \(E\subsetneq[0,n]\),
		\[
		\EHC_n(E)=|E|-\max\{i\in[0,n]:T_{n,i}\subseteq E\}.
		\]
	\end{corollary}
	\begin{proof}
		Let \(i_0=\max\{i\in[0,n]:T_{n,i}\subseteq E\}\in[0,\lceil,n/2\rceil]\).  So clearly, \(\EHC_n(E)\ge\EPC_n(E)=|E|-i_0\).  Let \(\{H_j:j\in[0,i_0-1]\}\) be the exact hyperplane cover of \(\ul{T_{n,i_0}}\), as given by Lemma~\ref{lem:GKN}.  For \(t\in E\setminus T_{n,i_0}\), let \(H'_t=\sum_{i\in[0,n]}X_i-t\).  Then it is immediate that \(\{H_j:j\in[0,i_0-1]\}\cup\{H'_t:t\in E\setminus T_{n,i_0}\}\) is an exact hyperplane cover of \(\ul{E}\), having size \(i_0+(|E|-2i_0)=|E|-i_0\).
	\end{proof}

	So the progress made in our work, as indicated in Table~\ref{tab:cube}, can be updated to the following.
	\begin{table}[h]
		\centering
		\bgroup
		\def\arraystretch{1.5}
		\begin{tabular}{|c|c|c|c|}
			\hline
			\tbf{Cover}&\tbf{Nontrivial}&\tbf{Proper}&\tbf{Exact}\\
			\hline
			\tbf{Hyperplane}&\cellcolor{SpringGreen}\(\HC_n\)&\cellcolor{Yellow}\(\PHC_n\)&\cellcolor{Yellow}\(\EHC_n\)\\
			\hline
			\tbf{Polynomial}&\cellcolor{SpringGreen}\(\PC_n\)&\cellcolor{Yellow}\(\PPC_n\)&\cellcolor{Yellow}\(\EPC_n\)\\
			\hline
		\end{tabular}
		\egroup
		\caption{For the Boolean cube \(\{0,1\}^n\) (Table~\ref{tab:cube} updated)}
	\end{table}
	
	\subsection{Tight examples}
	
	Here, for convenience, we quickly collect all the tight examples to the problems that we have solved.
	\begin{enumerate}[(1)]
		\item  Consider the Boolean cube \(\{0,1\}^n\), and any \(E\subsetneq[0,n]\).
		\begin{enumerate}[(a)]
			\item  \(\HC_n(E)=\PC_n(E)\).
			
			Let \(d_0=\min\{d\in[0,n]:E\tx{ is \(d\)-admitting}\}\), and \(i_0=\max\{i\in[0,d_0]:E\tx{ is \((d_0,i_0)\)-admitting}\}\).  Then we have \(E\cup T_{n,i_0}\ne[0,n]\), and necessarily, \(|E\setminus T_{n,i_0}|=d_0-i_0\).  A tight example is
			\[
			\{(X_1-X_2)\cdots(X_{2i_0-1}-X_{2i_0})\}\cup\bigg\{\sum_{i\in[0,n]}X_i-t:t\in E\setminus T_{n,i_0}\bigg\}.
			\]
			
			\item  \(\EHC_n(E)=\EPC_n(E)=\PPC_n(E)=\PHC_n(E)\).
			
			Let \(i_0=\max\{i\in[0,n]:T_{n,i}\subseteq E\}\).  A tight example is
			\[
			\bigg\{-(n-2i_0+j)X_{i-1-j}+\sum_{t\in[i_0-j]}X_t-j:j\in[0,i_0-1]\bigg\}\cup\bigg\{\sum_{t\in[0,n]}X_t-i:i\in E\setminus T_{n,i_0}\bigg\}.
			\]
		\end{enumerate}
	
		\item  Let \(G\) be an \SUt\,grid, and consider any \(E\subsetneq[0,N]\).  We have \(\PC_G(E)=\PPC_G(E)\).
		
		Let \(d_0=\min\{d\in[0,N]:E\tx{ is \(d\)-admitting}\}\), and \(i_0=\max\{i\in[0,d_0]:E\tx{ is \((d_0,i_0)\)-admitting}\}\).  Then we have \(E\cup T_{N,i_0}\ne[0,N]\), and necessarily, \(|E\setminus T_{N,i_0}|=d_0-i_0\).  A tight example is \(P(\mb{X})Q(\mb{X})\), where \(P(\mb{X})=\sum_{\gamma\in\ul{[0,i_0]}}c_\gamma\mb{X}^\gamma\) is a nontrivial solution of the linear system
		\[
		\sum_{\gamma\in\ul{[0,i_0]}}c_\gamma\mb{X}^\gamma(a)=0,\quad a\in\ul{[0,i_0-1]\cup\{N-i_0+1\}}.
		\]
		and
		\[
		Q(\mb{X})=\prod_{t\in E\setminus T_{N,i_0}}\bigg(\sum_{i\in[0,n]}X_i-t\bigg).
		\]
	\end{enumerate}

	\paragraph*{Acknowledgements.}  The author thanks
	\begin{itemize}
		\item  his graduate advisor Srikanth Srinivasan for valuable discussions throughout the gestation of this work, and unending support and encouragement.
		\item  Niranjan Balachandran for helpful comments on a preliminary version of this work.
		\item  Murali K. Srinivasan for a very enlightening discussion, as well as useful suggestions on the presentation of this work.
		\item  Lajos R\'onyai for pointers to some relevant literature.
		\item  Anurag Bishnoi for narrating the history of the hyperplane covering problems and the polynomial method, as well as for pointing out the recent work~\cite{ghosh-kayal-nandi-2022-alon-furedi}.
		\item  the anonymous referee for critical comments in an eagle-eyed review.
	\end{itemize}

	\raggedright
	\bibliographystyle{alpha}
	\bibliography{references}
	
	
\end{document}